\documentclass[a4paper,10pt]{amsart}
\usepackage{amsmath,amsthm,amssymb,latexsym,enumerate,xcolor,hyperref}
\usepackage{graphicx}
\usepackage[margin=2.5cm]{geometry}
\usepackage{mathtools}

\usepackage[dvipsnames]{xcolor}

\numberwithin{equation}{section}

\let\div\undefined
\DeclareMathOperator{\div}{div}

\begin{document}

	\newtheorem{thm}{Theorem}[section]
	\newtheorem{prop}[thm]{Proposition}
	\newtheorem{lem}[thm]{Lemma}
	\newtheorem{cor}[thm]{Corollary}
	\newtheorem{rem}[thm]{Remark}
	\newtheorem*{defn}{Definition}

	\newtheorem{definit}[thm]{Definition}
	\newtheorem{setting}{Setting}
	\renewcommand{\thesetting}{\Alph{setting}}
	
	\newcommand{\DD}{\mathbb{D}}
	\newcommand{\NN}{\mathbb{N}}
	\newcommand{\ZZ}{\mathbb{Z}}
	\newcommand{\QQ}{\mathbb{Q}}
	\newcommand{\RR}{\mathbb{R}}
	\newcommand{\CC}{\mathbb{C}}
	\renewcommand{\SS}{\mathbb{S}}

	\renewcommand{\theequation}{\arabic{section}.\arabic{equation}}

	\newcommand{\supp}{\mathop{\mathrm{supp}}}    
	
	\newcommand{\re}{\mathop{\mathrm{Re}}}   
	\newcommand{\im}{\mathop{\mathrm{Im}}}   
	\newcommand{\dist}{\mathop{\mathrm{dist}}}  
	\newcommand{\link}{\mathop{\circ\kern-.35em -}}
	\newcommand{\spn}{\mathop{\mathrm{span}}}   
	\newcommand{\ind}{\mathop{\mathrm{ind}}}   
	\newcommand{\rank}{\mathop{\mathrm{rank}}}   
	\newcommand{\Fix}{\mathop{\mathrm{Fix}}}   
	\newcommand{\codim}{\mathop{\mathrm{codim}}}   
	\newcommand{\conv}{\mathop{\mathrm{conv}}}   
	\newcommand{\epsi}{\mbox{$\varepsilon$}}
	\newcommand{\eps}{\mathchoice{\epsi}{\epsi}
		{\mbox{\scriptsize\epsi}}{\mbox{\tiny\epsi}}}
	\newcommand{\cl}{\overline}
	\newcommand{\pa}{\partial}
	\newcommand{\ve}{\varepsilon}
	\newcommand{\zi}{\zeta}
	\newcommand{\Si}{\Sigma}
	\newcommand{\cA}{{\mathcal A}}
	\newcommand{\cG}{{\mathcal G}}
	\newcommand{\cH}{{\mathcal H}}
	\newcommand{\cI}{{\mathcal I}}
	\newcommand{\cJ}{{\mathcal J}}
	\newcommand{\cK}{{\mathcal K}}
	\newcommand{\cL}{{\mathcal L}}
	\newcommand{\cN}{{\mathcal N}}
	\newcommand{\cR}{{\mathcal R}}
	\newcommand{\cS}{{\mathcal S}}
	\newcommand{\cT}{{\mathcal T}}
	\newcommand{\cU}{{\mathcal U}}
	\newcommand{\OM}{\Omega}
	\newcommand{\B}{\bullet}
	\newcommand{\ol}{\overline}
	\newcommand{\ul}{\underline}
	\newcommand{\vp}{\varphi}
	\newcommand{\AC}{\mathop{\mathrm{AC}}}   
	\newcommand{\Lip}{\mathop{\mathrm{Lip}}}   
	\newcommand{\es}{\mathop{\mathrm{esssup}}}   
	\newcommand{\les}{\mathop{\mathrm{les}}}   
	\newcommand{\nid}{\noindent}
	\newcommand{\pzr}{\phi^0_R}
	\newcommand{\pir}{\phi^\infty_R}
	\newcommand{\psr}{\phi^*_R}
	\newcommand{\pow}{\frac{N}{N-1}}
	\newcommand{\ncl}{\mathop{\mathrm{nc-lim}}}   
	\newcommand{\nvl}{\mathop{\mathrm{nv-lim}}}  
	\newcommand{\la}{\lambda}
	\newcommand{\La}{\Lambda}    
	\newcommand{\de}{\delta}    
	\newcommand{\fhi}{\varphi} 
	\newcommand{\ga}{\gamma}    
	\newcommand{\ka}{\kappa}   
	
	\newcommand{\core}{\heartsuit}
	\newcommand{\diam}{\mathrm{diam}}

	\newcommand{\lan}{\langle}
	\newcommand{\ran}{\rangle}
	\newcommand{\tr}{\mathop{\mathrm{tr}}}
	\newcommand{\diag}{\mathop{\mathrm{diag}}}
	\newcommand{\dv}{\mathop{\mathrm{div}}}
	
	\newcommand{\al}{\alpha}
	\newcommand{\be}{\beta}
	\newcommand{\Om}{\Omega}
	\newcommand{\na}{\nabla}
	
	\newcommand{\cC}{\mathcal{C}}
	\newcommand{\cM}{\mathcal{M}}
	\newcommand{\nr}{\Vert}
	\newcommand{\De}{\Delta}
	\newcommand{\cX}{\mathcal{X}}
	\newcommand{\cP}{\mathcal{P}}
	\newcommand{\om}{\omega}
	\newcommand{\si}{\sigma}
	\newcommand{\te}{\theta}
	\newcommand{\Ga}{\Gamma}
	
	\newcommand{\vV}{\mathbf{v}}
	\newcommand{\lbunu}{\ul{m}}
	\newcommand{\ca}{\ul{a}}
	\newcommand{\Vve}{\ul{\varepsilon}}
	\newcommand{\ur}{\ul{r}}

	\title[]{Optimal quantitative stability for the anisotropic Alexandrov soap bubble theorem and Heintze-Karcher inequality}
	
	\author{João Gonçalves da Silva}
	\address{School of Mathematical Sciences, Adelaide University, Adelaide SA 5005, Australia.}
	\email{joao.goncalvesdasilva@adelaide.edu.au}

	\author{Giorgio Poggesi}
	\address{School of Mathematical Sciences, Adelaide University, Adelaide SA 5005, Australia.}
	\email{giorgio.poggesi@adelaide.edu.au}

	\date{\today}
	
	\begin{abstract}
The present paper provides linear quantitative stability estimates for the anisotropic Alexandrov soap bubble theorem. The estimates are sharp and, to the authors’ knowledge, provide the first sharp quantitative version of Alexandrov’s theorem in the anisotropic setting.

Our estimates establish quantitative closeness to the Wulff shape, also providing sharp control of the closeness of the Gauss maps. As shown by the authors for the classical case, such a control can play a crucial role in achieving optimal profiles for arbitrary deviations and stronger measures of closeness. We also obtain sharp control of the so-called (anisotropic) excess.

The paper also provides the first sharp quantitative version of the anisotropic Heintze-Karcher inequality. For the Heintze-Karcher inequality, the stability estimates providing closeness of the Gauss maps and control of the excess are new even in the classical (isotropic) setting.
	\end{abstract}
\maketitle
\section{Introduction}

The present paper provides linear quantitative stability estimates for the anisotropic Alexandrov soap bubble theorem. The estimates are sharp and, to the authors’ knowledge, provide the first sharp quantitative version of Alexandrov’s theorem in the anisotropic setting.

The paper also provides the first sharp quantitative version of the anisotropic Heintze-Karcher inequality. For the Heintze-Karcher inequality, the stability estimates providing control of the closeness of the Gauss maps or of the excess are new even in the classical (isotropic) setting.

\subsection{The classical case: motivation and survey on the state of the art} 
It is well known from the calculus of variations that critical points of the classical isoperimetric problem have constant mean curvature boundaries. The celebrated Alexandrov Soap Bubble theorem provides a remarkable characterisation of these critical points.
In fact, Alexandrov’s Soap Bubble Theorem \cite{Al1,Al2} states that if a smooth\footnote{See also \cite{DM} for a recent generalisation to sets of finite perimeter.} bounded domain\footnote{In the present paper, {\it domain} means connected open set.} $\Omega$ in $\mathbb{R}^N$ (with $N \ge 2$) has boundary $\Gamma$ with constant mean curvature, then $\Omega$ is a ball (and $\Gamma$ is a sphere).
A natural question concerns the stability of this celebrated statement, which has attracted contributions from experts from disparate areas and provided fertile ground for research. In fact, this is not only a natural and fundamental question from the point of view of pure mathematics but also has fundamental applications in various areas, including, for instance, capillarity theory and the study of geometric flows. Indeed, the necessity of quantitative stability results for Alexandrov’s theorem naturally arises in the study of capillarity-type energy functionals \cite{CM} and in the study of the asymptotic behaviour of geometric flows (e.g., volume-preserving mean curvature flows and Mullins-Sekerka flows) \cite{JN, JMOS}. 

Investigating the stability of Alexandrov’s soap bubble theorem, one is naturally led to think of the following question:
if the boundary $\Gamma$ of a bounded and connected open set in $\RR^N$ has nearly constant mean curvature, must $\Gamma$ be close to a sphere? 
However, the answer to this question is negative in general, as {\it bubbling phenomena} can occur. In fact, configurations of connected sets converging to disjoint balls (of equal radii) can display arbitrary small deviations of their mean curvature from being constant (see, e.g., \cite{Butscher, ButscherMazzeo, Kapouleas_Annals, Kapouleas_JDG}). Nevertheless, such phenomena can be characterised and quantitatively studied. In fact, quantitative estimates of the closeness of a bounded open set $\Omega$ to a set given by the disjoint union of balls of equal radii were recently established in terms of various deviations of the mean curvature from some reference constant in three papers, as outlined in what follows. The first fundamental quantitative result in this direction is \cite{CM}, which deals with mean convex boundaries and uniform deviations of the mean curvature from being constant. Weakening the measure of the deviation in such a result is a delicate issue that is crucial in view of various applications, including those related to geometric flows, for which the necessity of results for $L^2$-type deviations typically arises. Progress in this direction was made in \cite{JN}, where a result for $L^{N-1}$ deviations was established, allowing applications to geometric flows in low dimensions (namely, in $\RR^N$ with $N \le 3$). Eventually, a result with $L^1$-type deviations (and even weaker) in any dimension was established in \cite{Po_JMPA}. From the qualitative point of view, it should be mentioned that, if $\Omega$ is assumed to be connected, dealing with $L^{N-1}$-type deviations or stronger ones (as in \cite{JN, CM}) leads to a limit configurations made of tangent balls of equal radii, whereas working with weaker deviations (as in \cite{Po_JMPA}) leads to limit configurations of disjoint (not necessarily tangent) balls of equal radii; in fact, the results in \cite{CM, JN} show that, when $\Omega$ is connected, the lengths of the possible ``necks’’ arising from the bubbling phenomena must tend to zero as the deviation of the mean curvature from being constant tends to zero, whereas the result in \cite{Po_JMPA} allows for thin but long necks.
Besides these foundational results that provided a deep understanding of the quantitative stability related to these bubbling phenomena, several works have focused on understanding the stability of Alexandrov’s theorem in the absence of bubbling, that is, under suitable bubbling-preventing assumptions as outlined in what follows. 
It is clear that, combining the above-mentioned results from \cite{CM, JN, Po_JMPA} with any bubbling-preventing assumption (for instance an isoperimetric bound, a nearly sphericity assumption, or a uniform regularity assumption) immediately leads to quantitative estimates of closeness to a single ball. Nevertheless, under these kinds of assumptions one has more tools and can achieve stronger control and sharp rates of stability as outlined in what follows. (Sharp rates of stability are not available yet in the general setting where bubbling is allowed.)

A natural bubbling-preventing assumption is given by assuming that the set is nearly spherical. In this direction, sharp quantitative versions of Alexandrov’s theorem (for sets that are Lipschitz-close to a sphere) were established in \cite[Theorems 1.8 and 1.10]{KM}, and subsequently in \cite{MPS, FZ_2023}. Such a nearly spherical approach was inspired by the seminal work of Fuglede \cite{Fuglede_TAMS} on the stability of the isoperimetric problem. 

Another alternative natural bubbling-preventing condition is given by assuming some uniform regularity of the domain. A result in this direction was established in \cite{CV}, which deals with uniform deviations by using the method of the moving planes. The estimate is the following
\begin{equation}\label{eq:CV result}
\inf_{z \in \RR^N} \left( \rho_e (z) - \rho_i (z) \right) \le C \left( \max_{p \in \Gamma} \cM(p) -\min_{p \in \Gamma} \cM(p) \right) ,
\end{equation}
where  $\cM$ denotes the mean curvature of $\Gamma$,
\begin{equation}\label{eq:intro:def rhoei}
\rho_e (z) := \max_{x \in \Gamma} |x-z| , \quad \rho_i (z) := \min_{x \in \Gamma} |x-z|  ,
\end{equation}
and the constant $C$ depends on the $C^{1,1}$ regularity of $\Gamma$ and the $N-1$-dimensional measure of $\Gamma$. Actually, the result in \cite{CV} states dependency on the radius of the uniform interior and exterior touching ball condition, which can be shown to be equivalent to $C^{1,1}$ regularity (see, e.g., \cite{ABMMZ}).\\
The uniform-type deviation $\max_{p \in \Gamma} \cM(p) -\min_{p \in \Gamma} \cM(p) $ was successfully replaced with weaker deviations in \cite{MP_JAM, MP_IUMJ, MP_CVPDE_2020, MP_MinE, GdSPo_CVPDE} by using different approaches via integral identities and inequalities.
In fact, linear stability estimates were obtained in any dimension for $L^2$ deviations at the cost of replacing the uniform closeness on the left-hand side of \eqref{eq:CV result} with a weaker norm, namely, $L^1$ closeness of $\Omega$ to a ball \cite[Theorem 4.6]{MP_IUMJ}, which can be improved with an $L^2$ closeness of $\Gamma$ to a sphere \cite[(i) in Theorem 1.3]{Po_TAMS}, which can be further improved to $L^2$ closeness of the Gauss maps \cite[$(3.10)$ in Theorem 3.9]{MP_MinE}. The last two results (\cite[(i) in Theorem 1.3]{Po_TAMS} and \cite[$(3.10)$ in Theorem 3.9]{MP_MinE}) read as follows: there exists a point $z \in \Om$ such that
\begin{equation}\label{eq:L^2 Gauss map}
\left\nr |x-z| - \frac{1}{\cM_{0}} \right\nr_{L^2(\Gamma)} + \left\nr \frac{ \nu}{\cM_{0}} - (x-z) \right\nr_{L^2(\Gamma)} \le C \, \nr \cM - \cM_{0} \nr_{L^2(\Gamma)},
\end{equation}
where $\nu$ denotes the exterior unit normal to $\Gamma$, and we have set
\begin{equation}\label{eq:intro:def H_0}
\cM_{0} := \frac{ P(\Om)}{N |\Omega|},
\end{equation}
where $|\Omega|$ denotes the $N$-dimensional measure of $\Omega$ and $P(\Om)$ denotes the perimeter of $\Om$. The constant $C$ in \eqref{eq:L^2 Gauss map} only depends on $N$, the diameter of $\Omega$, and the $C^{1,1}$ regularity of $\Gamma$.
The optimality of those linear stability profiles can be directly checked using ellipsoids.

The results in \cite{MP_JAM, MP_IUMJ, MP_CVPDE_2020, MP_MinE} also included estimates of uniform closeness of $\Gamma$ to a sphere (i.e., $\rho_e (z) -\rho_i (z)$, where $\rho_e (z)$ and $\rho_i (z)$ are as defined in \eqref{eq:intro:def rhoei}) for $L^2$ (and $L^1$-type) deviations of the mean curvature from being constant. In particular, for $L^2$ deviations, the stability profiles concerning the uniform closeness in \cite{MP_JAM, MP_IUMJ, MP_CVPDE_2020, MP_MinE} were almost optimal; more precisely, they were optimal in low dimensions ($N \le 3$), but not optimal for $N \ge 4$.
Optimal stability estimates in any dimension were finally established by the authors in \cite{GdSPo_CVPDE}, by leveraging a dimension reduction argument and Gagliardo-Nirenberg-type and Sobolev-type inequalities. More generally, for any $r \ge 2$ (more generally, for any $r > 1$ satisfying $r \ge \frac{2N - 2}{N+1}$), $k\ge 1$ and $0< \al\le 1$, \cite{GdSPo_CVPDE} provides the following stability profiles for the class of uniform $C^{k,\al}$ domains, in any dimension and for any $L^r$-type deviation, namely,
\begin{equation}\label{eq:intro:CVPDE_GdSPo}
	\rho_e (z) - \rho_i (z) \le C
	\begin{cases}
		\nr \cM- \cM_0 \nr_{L^r(\Gamma)}, \quad & \text{ if } N< 2r + 1
		\\
		\nr \cM- \cM_0 \nr_{L^r(\Gamma)} \max \left\lbrace\frac{1}{k + \alpha}    \log\left( \frac{1}{ \nr \cM- \cM_0 \nr_{L^r(\Gamma)} } \right) , 1 \right\rbrace , \quad & \text{ if } N= 2r + 1
		\\
		\nr \cM- \cM_0 \nr_{L^r(\Gamma)}^{\tau_{k, \alpha , N, r}} , \quad & \text{ if } N> 2r + 1
	\end{cases}
\end{equation}
for some $z \in \Omega$. Here, $\rho_e (z)$ and $\rho_i (z)$ are as defined in \eqref{eq:intro:def rhoei}, $\cM_0$ is as defined in \eqref{eq:intro:def H_0}, and
\begin{equation*}
	\tau_{k, \alpha , N, r} := \frac{k +  \alpha}{k +  \alpha + \frac{N-1-2r}{r}}. 
\end{equation*}
The constant $C$ in \eqref{eq:intro:CVPDE_GdSPo} only depends on $N$, $r$, the diameter of $\Omega$, and the $C^{k,\al}$ regularity of $\Gamma$.
The optimality of the linear profiles for $N<2r+1$ can be easily verified by direct computations with ellipsoids.
Note that the profile in \eqref{eq:intro:CVPDE_GdSPo} is not linear in high dimensions, but the stability exponent $\tau_{k, \alpha, N, r}$ is shown to be optimal in \cite{GdSPo_CVPDE} with explicit examples. 

As mentioned, a crucial step that allowed to achieve the optimal profiles in \eqref{eq:intro:CVPDE_GdSPo} was to notice that if the deviation $\nr \cM- \cM_0 \nr_{L^r(\Gamma)}$ -- and we are working within a class of uniformly $C^{k,\al}$ ($k \ge 1$ and $0< \al \le 1$) domains, which prevents bubbling phenomena --, then $\Omega$ must be a nearly spherical set. This qualitative information was established in \cite{GdSPo_CVPDE} by crucially leveraging \eqref{eq:L^2 Gauss map}.

Other results were recently established under alternative bubbling-preventing conditions involving, e.g., isoperimetric bounds \cite{JMOS}, the total Willmore energy \cite{BiZhou_2022}, or the so-called excess \cite{FZ_2026}.

\subsection{The anisotropic case}
In the present paper we provide the first sharp stability result for the anisotropic extension of Alexandrov’s soap bubble theorem. The available literature on the topic in the anisotropic setting is outlined in what follows. A fundamental qualitative result on bubbling phenomena was established in \cite{DMMN} for $L^2$ deviations of the anisotropic mean curvature from being constant, and under a mean convexity assumption. A quantitative stability result of closeness to the single Wulff-ball was later established in \cite{ScheuerZhang_JFA}, for $L^1$ deviations. 
Here, we are providing the first optimal stability profile, by establishing the anisotropic extension of \eqref{eq:L^2 Gauss map}. Not only is this the first optimal stability estimate in the anisotropic setting, but, as shown in \cite{GdSPo_CVPDE} for the classical case, such a control is crucial for achieving optimal profiles for arbitrary deviations and stronger measures of closeness as in \eqref{eq:intro:CVPDE_GdSPo}.
In the present paper, not only do we provide the full extension of \eqref{eq:L^2 Gauss map} to the anisotropic setting, but we also establish sharp control of the so-called anisotropic excess (see (iii) of Remark \ref{rem: on the three summands} for the definition), which has potential for future applications. In fact, having control of the excess has significant consequences for the stability issue, as recently shown in the isotropic setting in \cite{FZ_2026}.

The present paper also provides the first sharp quantitative version of the anisotropic Heintze-Karcher inequality. For the Heintze-Karcher inequality, the stability estimates providing closeness of the Gauss maps and control of the excess are new even in the classical (isotropic) setting.

 \subsubsection{Main results}\label{subsec:Main results}

Throughout this paper, $H\in C^{2,\alpha}\left(\RR^N\setminus\{0\}\right)$ (for some $\alpha\in(0,1)$) will be a {\it Finsler norm}, i.e., $H$ satisfies the following properties: 
\begin{equation}\label{AddedAssumpFinslernorm:convex}
     H \text{ is convex;}
\end{equation}
\begin{equation}\label{Assump:positivity of Finsler norm}
    H(\xi)\geq 0,\,\forall\, \xi\in \mathbb{R}^N,\,\, \text{and}\,\, H(\xi) = 0 \text{  if and only if  } \xi =0;
\end{equation}
\begin{equation}\label{Assump:Homogeneity assumption}
    H(t\xi) = tH(\xi), \,\,\text{for all}\,\, \xi \in \mathbb{R}^N \,\, \text{and}\,\, t\geq 0.
\end{equation}
Property \eqref{Assump:Homogeneity assumption} says that $H$ is positively $1$-homogeneous.\\
Note that a Finsler norm does not need to be a norm in the classical sense, as it is not necessarily symmetric. Of course, any norm in the classical sense is a Finsler norm.

We also assume that
$H$ satisfies the following ellipticity condition:
\begin{equation}\label{Assump:ellipticity assumption}
 V\coloneqq \frac{1}{2}H^2 \, \text{ satisfies } \,
    \lambda |\eta|^2\leq \eta^{T}\nabla^{2}_{\xi}V(\xi)\eta\leq \Lambda |\eta|^2,\,\, \text{for every}\,\, \xi\in \mathbb{S}^{N-1}\,\,\text{and every}\,\, \eta\in\RR^N,
\end{equation}
for some $0<\lambda<\Lambda$.

The dual Finsler norm of $H$ in $\mathbb{R}^N$, $H_0: \mathbb{R}^N\rightarrow [0,\infty)$ is defined by
	\begin{equation}\label{eq:dual norm}
		H_0(x):= \max_{\xi\neq 0}\frac{\langle x\,,\,\xi\rangle}{H(\xi)},
	\end{equation}
	where $\langle \cdot \,,\,\cdot\rangle$ denotes the standard inner product in $\mathbb{R}^N$. 
    
    The Wulff ball with respect to $H_0$ centered at $x$ and radius $r$, denoted $B_{r}^{H_0}(x)$, is defined as:
	\begin{equation}\label{eq:norm in the dual norm}
		B_{r}^{H_0}(x) = \left\{y\in \mathbb{R}^N:\, H_0(y-x)<r\right\},
	\end{equation}
    the Wulff sphere centered at $x$ and radius $r$, $\mathcal{W}_{r}(x)$, is given by 
    $$
    \mathcal{W}_{r}(x) = \partial B^{H_0}_{r}(x) = \left\{y\in \mathbb{R}^N:\, H_0(y-x)=r\right\}.
    $$
    
For a sufficiently smooth domain $\Omega$ with boundary $\Ga$ and outward unit normal $\nu$, its anisotropic outward normal is defined as 
    \begin{equation}\label{eq: def_anisotropic normal}
    \nu_H := \na_{\xi} H \left( \nu \right).
    \end{equation}

Let $\Omega\subset \RR^N$ be a bounded domain with sufficiently regular boundary $\Gamma$.
The anisotropic perimeter of $\Omega$ is defined as 
    \begin{equation}\label{eq: def_anisotropic perimeter}
    P_H(\Omega) := \int_{\Gamma} H(\nu)\,d\mathcal{H}^{N-1}(x),
    \end{equation}
and the anisotropic mean curvature of $\Ga$, with respect to $\nu$, $\mathcal{M}^{H}$, is defined as
    \begin{equation}\label{eq:anisotropic mean curvature}
        \mathcal{M}^{H} = \frac{1}{N-1}\div_{\Ga}\left(\nu_{H}\right),
    \end{equation}
   where $\div_{\Ga}$ denotes the tangential divergence operator.

\bigskip

Our first main theorem is the following optimal stability result for the anisotropic soap bubble theorem.

\begin{thm}[Optimal stability for the anisotropic soap bubble theorem]\label{thm:W12 stability}
 Let $H\in C^{2,\alpha}\left(\RR^N\setminus\{0\}\right)$ be a Finsler norm satisfying \eqref{Assump:ellipticity assumption}. Let $\Omega\subset \RR^N$ be a bounded domain with boundary $\Gamma$ of class $C^{1,1}$. Set 
\begin{equation}\label{eq:reference curvature}
     \mathcal{M}_{0}^{H} \coloneqq \frac{P_{H}(\Om)}{N\left|\Om\right|}.
 \end{equation}
 Then there exists $z\in \RR^N$, such that 
 \begin{equation}\label{eq:optimal W12 stability}
     \left\|\frac{1}{\mathcal{M}_{0}^{H}}-H_{0}(x-z)\right\|_{L^{2}\left(\Gamma\right)} + \left\|\frac{1}{\mathcal{M}_{0}^{H}}\nu_{H}- (x-z) \right\|_{L^{2}\left(\Gamma\right)}+\left\|\nu- \frac{\nabla H_{0}(x-z)}{\left|\nabla H_{0} (x-z)\right|} \right\|_{L^{2}\left(\Gamma\right)}\leq C \left\|\mathcal{M}-\cM_{0}^{H}\right\|_{L^{2}(\Ga)},
 \end{equation}
 where $C$ is a positive constant dependent only on $N$, the parameters $\Lambda$, $\lambda$ appearing in \eqref{Assump:ellipticity assumption}, the diameter of $\Om$ and the $C^{1,1}$ regularity of $\Ga$.
\end{thm}

\begin{rem}\label{rem: on the three summands}
(i) The first summand on the left-hand side of \eqref{eq:optimal W12 stability} provides sharp linear control of the $L^2$ distance of $\Gamma$ to the Wulff shape centred at $z$ with radius $\frac{1}{\mathcal{M}_{0}^{H}}$, that is,
\begin{equation*}
\left\|\frac{1}{\mathcal{M}_{0}^{H}}-H_{0}(x-z)\right\|_{L^{2}\left(\Gamma\right)}
\leq C \left\|\mathcal{M}-\cM_{0}^{H}\right\|_{L^{2}(\Ga)}.
\end{equation*}

(ii) The second summand on the left-hand side of \eqref{eq:optimal W12 stability}
provides sharp linear control of the $L^2$ closeness of the (anisotropic) Gauss maps. More precisely, it provides sharp linear control on the $L^2$ distance between the anisotropic exterior normal $\nu_H$ to $\Gamma$ and $(x-z) \mathcal{M}_{0}^{H} $, that is,
\begin{equation*}
\left\|\frac{1}{\mathcal{M}_{0}^{H}}\nu_{H}- (x-z) \right\|_{L^{2}\left(\Gamma\right)}
\leq C \left\|\mathcal{M}-\cM_{0}^{H}\right\|_{L^{2}(\Ga)}.
\end{equation*}

(iii) It is clear that the control of these first two summands already provides the full anisotropic generalisation of \eqref{eq:L^2 Gauss map}.

Here, we further achieve sharp linear control of the {\it anisotropic excess} -- which is the third summand on the left-hand side of \eqref{eq:optimal W12 stability} --, namely:
\begin{equation*}
\left\|\nu- \frac{\nabla H_{0}(x-z)}{\left|\nabla H_{0} (x-z)\right|} \right\|_{L^{2}\left(\Gamma\right)}
\leq C \left\|\mathcal{M}-\cM_{0}^{H}\right\|_{L^{2}(\Ga)}.
\end{equation*}
The excess measures the $L^2$ distance between the classical exterior unit normal at $x\in \Gamma$ and the classical exterior unit normal $\frac{\nabla H_{0}(x-z)}{\left|\nabla H_{0} (x-z)\right|}$ to the Wulff shape centred at $z$ and passing through $x$.\\
The optimality of the linear stability profile in the last three inequalities, and hence in \eqref{eq:optimal W12 stability} can be easily verified in the isotropic case by direct computation with ellipsoids.
\end{rem}

Note that our result is global in the sense that no smallness assumption on the deviation $\left\|\mathcal{M}-\cM_{0}^{H}\right\|_{L^{2}(\Ga)}$ is required. 
It is clear that, in the special case where the deviation is zero, Theorem \ref{thm:W12 stability} yields the rigidity statement $\Om = B_{r}^{H_0}(z)$ and $\Ga = \mathcal{W}_{r}(z)$. In particular, if the mean curvature $\mathcal{M}$ is assumed to be constant on $\Gamma$, then the Minkowski identity (Lemma~\ref{lemma:Anisotropic Minkowski}) implies that $\mathcal{M} = \cM_{0}^{H}$ on $\Gamma$, so the deviation is zero and Theorem \ref{thm:W12 stability} therefore recovers
the anisotropic extension of Alexandrov's Soap Bubble Theorem (originally obtained in \cite{HLMG}), as stated in the following corollary.

\begin{cor}\label{cor:Anisotropic Bubbles}
    Let $H\in C^{2,\alpha}\left(\RR^N\setminus\{0\}\right)$ be a Finsler norm satisfying \eqref{Assump:ellipticity assumption}. Let $\Omega\subset \RR^N$ be a bounded domain with boundary $\Gamma$ of class $C^{1,1}$. 
    If $\cM^{H}$ is constant, then $\cM^{H} = \cM_{0}^{H}$, where $\cM_{0}^{H}$ is as in \eqref{eq:reference curvature}, and for some $z\in\RR^N$
    $$
    \Om = B_{r}^{H_0}(z) \quad \text{and}\quad \Ga = \mathcal{W}_{r}(z),
    $$
    where $r= \frac{1}{\cM_{0}^{H}}$.
\end{cor}

\medskip


Our second main theorem establishes an optimal quantitative stability result for the anisotropic Heintze-Karcher inequality. Before stating the theorem, we briefly recall some relevant developments in the literature.

The classical Heintze-Karcher inequality \cite{HK_1978} states that if $\Omega$ is a bounded domain whose boundary $\Gamma$ is smooth enough and has positive mean curvature $\cM >0$, then
\begin{equation}
\int_{\Gamma} \frac{1}{\cM} d \mathcal{H}^{N-1} \ge N | \Omega | ;
\end{equation}
moreover, equality holds if and only if $\Omega$ is a ball.

Quantitative stability for the classical Heintze-Karcher inequality was first established in \cite{MP_JAM}. The following optimal quantitative version was later achieved in \cite[(ii) in Theorem 1.3]{Po_TAMS}:
\begin{equation}\label{eq:intro:HK Po TAMS}
\left\| |x-z|^2 - R^2 \right\|_{L^2(\Gamma)}^2 \le C \left( \int_{\Gamma} \frac{1}{\cM} d \mathcal{H}^{N-1} - N | \Omega | \right) ,
\end{equation}
for some suitable choice of $z \in \mathbb{R}^N$ and $R >0$
(see \cite{Po_TAMS} for details).

Given a Finsler norm $H$, the anisotropic generalisation of the Heintze-Karcher inequality states that if $\Omega$ is a bounded domain whose boundary $\Gamma$ is smooth enough and has positive anisotropic mean curvature $\mathcal{M}^H > 0$, then
\begin{equation}\label{eq:intro_anisotropic HK}
\int_{\Gamma} \frac{H(\nu)}{\mathcal{M}^H} d \mathcal{H}^{N-1} \ge N | \Omega | ;
\end{equation}
moreover, equality holds if and only if $\Omega$ is a Wulff ball.
Various proofs of this inequality were obtained in \cite{HLMG, MX_2013, XZ_2017, DMMN}.

Here, we provide the first optimal quantitative version of \eqref{eq:intro_anisotropic HK}. Not only our result provides the generalisation of \eqref{eq:intro:HK Po TAMS} to the anisotropic setting, but it also provides sharp estimates for the excess and the Gauss maps, which turn out to be new even in the classical (isotropic) setting.


\begin{thm}[Optimal stability for the anisotropic Heintze-Karcher inequality]\label{thm:stability of Heintze-Karcher}
   Let $H\in C^{2,\alpha}\left(\RR^N\setminus\{0\}\right)$ be a Finsler norm satisfying \eqref{Assump:ellipticity assumption}.
    Let $\Om\subset \RR^N$ be a bounded domain with boundary $\Ga$ of class $C^{1,1}$ with $\cM^{H}>0$.\newline
    Then, there exist a positive constant $R$ and a point $z\in \RR^N$, such that 
    \begin{equation}\label{eq:W 1,2 stability for HK}
    \begin{split}
        \left\|H_{0}(x-z)-R\right\|_{L^{2}(\Ga)}^2 + \left\|R\nu_{H}-(x-z) \right\|_{L^{2}\left(\Gamma\right)}^2 & +\left\|\nu - \frac{\nabla H_{0}(x-z)}{\left|\nabla H_{0}(x-z)\right|}\right\|_{L^2(\Ga)}^2 \\
        &\leq C \left(\int_{\Ga}\frac{H(\nu)}{\cM^{H}}\,d\mathcal{H}^{N-1}-N\left|\Om\right|\right),
    \end{split}
    \end{equation}
    where $C$ is a positive constant dependent on $N$, the parameters $\lambda$ and $\Lambda$ appearing in \eqref{Assump:ellipticity assumption}, the diameter of $\Om$ and the $C^{1,1}$ regularity of $\Ga$.
\end{thm}

\begin{rem}
An analogous of Remark \ref{rem: on the three summands} holds true. In fact, the three summands on the left-hand side of \eqref{eq:W 1,2 stability for HK} provide sharp control of the $L^2$-distance of $\Gamma$ to the Wulff shape $\mathcal{W}_{R}(z)$,
\begin{equation*}
 \left\|H_{0}(x-z)-R\right\|_{L^{2}(\Ga)}^2  
 \leq C \left(\int_{\Ga}\frac{H(\nu)}{\cM^{H}}\,d\mathcal{H}^{N-1}-N\left|\Om\right|\right) ,
\end{equation*}
sharp control for the anisotropic Gauss map, namely
\begin{equation*}
\left\|R\nu_{H}-(x-z) \right\|_{L^{2}\left(\Gamma\right)}^2  
 \leq C \left(\int_{\Ga}\frac{H(\nu)}{\cM^{H}}\,d\mathcal{H}^{N-1}-N\left|\Om\right|\right) ,
\end{equation*}
and sharp control for the anisotropic excess, namely
\begin{equation*}
\left\|\nu - \frac{\nabla H_{0}(x-z)}{\left|\nabla H_{0}(x-z)\right|}\right\|_{L^2(\Ga)}^2
\leq C \left(\int_{\Ga}\frac{H(\nu)}{\cM^{H}}\,d\mathcal{H}^{N-1}-N\left|\Om\right|\right).
\end{equation*}

The last two estimates are new even in the isotropic setting.\\
 The optimality of the stability exponent in the last three inequalities, and hence in \eqref{eq:W 1,2 stability for HK}, can be easily verified in the isotropic case by direct computation with ellipsoids. 
\end{rem}
 
As a direct corollary of this result, we clearly have the anisotropic Heintze-Karcher inequality along with the characterisation of its equality case. 
\begin{cor}
     Let $H\in C^{2,\alpha}\left(\RR^N\setminus\{0\}\right)$ be a Finsler norm satisfying \eqref{Assump:ellipticity assumption}.
    Let $\Om\subset \RR^N$ be a bounded domain with boundary $\Ga$ of class $C^{1,1}$ with $\cM^{H}>0$. 
    Then, the anisotropic Heintze-Karcher inequality
    $$
    \int_{\Ga}\frac{H(\nu)}{\cM^{H}}\,d\mathcal{H}^{N-1}-N\left|\Om\right| \ge 0,
    $$
    holds. Moreover, if 
    $$
    \int_{\Ga}\frac{H(\nu)}{\cM^{H}}\,d\mathcal{H}^{N-1}-N\left|\Om\right|=0,
    $$
    then
    
$$
    \Om = B_{R}^{H_0}(z) \quad \text{and}\quad \Ga = \mathcal{W}_{R}(z),
$$
for some $z\in\RR^N$ and $R>0$. 
\end{cor}
 
\subsection{Organisation of the paper}  
	The rest of the paper is organised as follows. 

	In Section \ref{sec:Preliminaries} we introduce the setting and recall some basic tools that will be used in the proof of our main results.
	Section \ref{sec:Fundamental identity for ASBT} is devoted to the proofs of some fundamental integral identities, namely Theorem \ref{thm:fundamental identity for ASBT} and Theorem \ref{thm:fundamental identity for Heintze-karcher}, which serve as the starting points for our analysis leading to establish Theorems \ref{thm:W12 stability} and \ref{thm:stability of Heintze-Karcher}, respectively.
    In Section \ref{Sec: W12 stability for ASBT}, we prove Theorem~\ref{thm:W12 stability}.
	We conclude the paper with Section \ref{sec:stability for AHK}, which is devoted to the proof of Theorem~\ref{thm:stability of Heintze-Karcher}.

\section{Setting and Background Result}\label{sec:Preliminaries}

	\subsection{On Finsler norms}
    Let $H\in C^{2,\alpha}\left(\RR^N\setminus\{0\}\right)$ be a Finsler norm (i.e., it satisfies \eqref{AddedAssumpFinslernorm:convex}, \eqref{Assump:positivity of Finsler norm}, \eqref{Assump:Homogeneity assumption}) satisfying \eqref{Assump:ellipticity assumption}.\newline
    Under these assumptions, it is clear that $B_{1}^{H}(0)$ is uniformly convex and has boundary of class $C^{2,\alpha}$. Moreover, recalling the definition of $H_0$ in \eqref{eq:dual norm}, we see that $H$ is the support function of the convex set $B_{1}^{H_0}(0)$ and
	\begin{equation*}
		H(\xi) = \max_{x\neq 0} \frac{\langle x\,,\, \xi \rangle}{H_0(x)}, \quad \text{for} \,\, \xi\in \RR^{N}.
	\end{equation*}
	Therefore, we also have that $H_0$ is the support function of the convex set $B_{1}^{H}(0)$. In particular, we have that $H_0 \in C^{1}(\RR^N\setminus \{0\})$ and that $B_{1}^{H_{0}}(0)= \left\{x\in\RR^N:\,H_{0}(x)<1\right\}$ is also uniformly convex. We refer the reader to \cite[Chapter $1$ and Chapter $2$]{Schneider} for additional details.\newline

    Due to \eqref{Assump:Homogeneity assumption} and Euler's homogeneous function theorem, the following identities hold
    \begin{align}
        & \nabla_{\xi}H(t\xi) = \nabla_{\xi}H(x),\quad \text{for } \xi \in \RR^N\setminus \{0\} \text{ and } t >0;\label{eq:derivative of norm}\\
        &\Big\langle\nabla_{\xi}H(\xi)\,,\, \xi\Big\rangle = H(\xi), \quad \text{for } \xi \in \RR^N\setminus \{0\};\label{eq:inner product grad of norm with point}\\
        &\nabla^{2}_{\xi}H^2(\xi)\xi = \nabla_{\xi}H^2(\xi) \quad \text{for } \xi \in \RR^N\setminus \{0\};\label{eq:homogeneity of second order for H^2}\\
        &\xi^T\nabla_{\xi}H^2(\xi)\xi = 2 H^2(\xi) \quad \text{for } \xi \in \RR^N\setminus\{0\}.\label{eq:homogeneity of first order for H^2}
    \end{align}
    Here, we extend the left-hand side of \eqref{eq:inner product grad of norm with point} and \eqref{eq:homogeneity of first order for H^2}, by continuity, to zero at $\xi = 0$.
    
    Using \eqref{eq:homogeneity of first order for H^2}, we see that taking $\eta=\xi \in \SS^{N-1}$ in \eqref{Assump:ellipticity assumption} we see that $\sqrt{\lambda}\leq H(\eta)\leq \sqrt{\Lambda}$, for every $\eta\in\SS^{N-1}$. So, by the positive $1-$homogeneity of $H$ and the definition of $H_0$, we conclude that, for all $x,\xi\in \RR^N$
    \begin{equation}\label{eq:norm equivalency}
		 \sqrt{\lambda}|\xi|\leq H(\xi)\leq \sqrt{\Lambda}|\xi|\quad \text{and}\quad  \frac{1}{\sqrt{\Lambda}}|x|\leq H_{0}(x)\leq \frac{1}{\sqrt{\lambda}}|x|.
	\end{equation}
    Throughout the paper, we will employ the notation 
    $$V_{0}\coloneqq \frac{1}{2}H_{0}^2.$$ 
The following lemma will be useful throughout the paper and shows that the maps $\nabla_{\xi}V$ and $\nabla V_{0}$ are inverses of each other.

	\begin{lem}[Lemma $3.1$, \cite{CS}]\label{lem:inversion lemma}
		Let $H\in C^{2,\alpha}\left(\RR^N\setminus \{0\}\right)$ be a Finsler norm satisfying \eqref{Assump:ellipticity assumption}. Let $H_0$ be the dual Finsler norm as defined in \eqref{eq:dual norm}. Then 
		\begin{align}
				&H(\nabla H_0(x)) = 1 \text{ for all } x\in \RR^N\setminus\{0\},\label{eq:property 1 grad of norms}\\
				& H_0(\nabla_{\xi} H(\xi)) = 1 \text{ for all } \xi\in \RR^N\setminus\{0\}.\label{eq:property 2 grad of norms}
		\end{align}
		Furthermore, the map $\nabla_{\xi} V:\RR^N\rightarrow \RR^N$ is invertible and 
		\begin{equation}\label{eq:inverse maps}
			\left(\nabla_{\xi} V\right)^{-1} = \nabla V_0. 
		\end{equation}
		Here, $ \nabla_{\xi} V$ and $ \nabla V_0$ are taken to be zero at zero.
	\end{lem}
     Although \cite[Lemma $3.1$]{CS} was stated for norms, an inspection of the proof shows that the result remains valid for Finsler norms.
    
As a direct consequence of this lemma we see that if $H\in C^{2,\alpha}(\RR^N\setminus\{0\})$, then $H_{0}\in C^{2,\alpha}(\RR^N\setminus\{0\})$. Moreover, from the positive $2$-homogeneity of $V$ and $V_{0}$, we see that $\nabla_{\xi}V,\nabla V_{0}\in {\rm Lip}(\RR^N;\RR^N)$. In addition, \eqref{Assump:ellipticity assumption}, Lemma~\ref{lem:inversion lemma}, and the positive $0$-homogeneity of $\nabla^2 V_0$ imply that 
\begin{equation}\label{eq:ellipticity of V_0}
    \Lambda^{-1}|\xi|^2\leq \xi^T\nabla^{2}V_{0}(\eta)\xi\leq \lambda^{-1}|\xi|^2,\,\, \text{for any}\,\, \eta\in \RR^N\setminus\{0\}\,\,\text{and any}\,\, \xi \in\RR^N.
\end{equation}
As a consequence 
$$
N^{1/2}\sqrt{\lambda}\leq \left\|\nabla^2 V\right\|_{L^{\infty}(\RR^N)}\leq N^{1/2}\Lambda^{1/2},
$$ 
and 
$$
N^{1/2}\Lambda^{-1/2}\leq \left\|\nabla^2 V_{0}\right\|_{L^{\infty}(\RR^N)}\leq N^{1/2}\lambda^{-1/2}.
$$
Therefore, we have that 
\begin{equation}\label{eq:inversion formulas}
    \begin{split}
    &\left| \nabla_{\xi}V(\xi_{1}) - \nabla_{\xi}V(\xi_{2})\right|\leq \left\|\nabla_{\xi}^{2}V\right\|_{L^{\infty}(\RR^N)}\left|\xi_1-\xi_2\right|\leq N^{1/2}\Lambda^{1/2}\left|\xi_1-\xi_2\right|,\quad \text{for all}\quad \xi_{1}, \xi_{2}\in \RR^N,\\
    &\left|\nabla V_{0}(x_{1}) - \nabla V_{0}(x_{2})\right|\leq \left\|\nabla^{2} V_{0}\right\|_{L^{\infty}(\RR^N)}\left|x_1-x_2\right|\leq N^{1/2}\lambda^{-1/2}\left|x_1-x_2\right|,\quad \text{for all}\quad x_{1}, x_{2}\in \RR^N.
    \end{split}
\end{equation}

\subsection{On the Anisotropic Torsion Problem and the Anisotropic Mean Curvature}\label{subsec: Weak solutions of Anisotropic torsion}

We start this subsection by presenting the definition of $C^{k,\gamma}$ domains that we adopt in this paper (see e.g., \cite[Page 94]{GT}).

    \begin{definit}\label{def:regularity of a domain}
        Given an integer $k\geq 1$ a real number $\gamma\in [0,1]$ and a bounded domain $D \subset \mathbb{R}^N$ ($N\geq 2$) with boundary $\partial D$. We say that $\partial D$ is of class $C^{k,\gamma}$, if there exist positive constants $K$ and $r$ such that for any $x\in \partial D$, there exists a one-to-one map 
            \begin{equation*}
                \psi_{x}: B_{r}(x) \rightarrow D \subset \mathbb{R}^N,  
            \end{equation*}
            such that 
            \begin{align*}
                &\psi_{x}(B_{r}(x)\cap D) \subset \mathbb{R}_{+}^N , \quad \psi_{x}(B_{r}(x)\cap \partial D) \subset \partial\mathbb{R}_{+}^N,\\
                &\|\psi_x\|_{C^{k,\gamma}\left(B_{r}(x)\right)}\le K \quad \text{and}\quad \|\psi_{x}^{-1}\|_{C^{k,\gamma}(D)}\leq K.
            \end{align*}
            Here above, $B_{r}(x)$ denotes the $N-$dimensional euclidean ball centered at $x\in \mathbb{R}^N$ with radius $r$. 

            In particular, for every $x\in \Ga$, $B_{r}(x)\cap \Ga$ can be written as the graph of a $C^{k,\gamma}$ function of $N-1$ variables, whose $C^{k,\gamma}$ norm is bounded above by $K$.
    \end{definit} 
    Throughout the paper, when we say that a constant depends on the $C^{k,\gamma}$ regularity of $\Ga$, it means that it can be estimated in terms of the constants $K$ and $r$ in Definition \ref{def:regularity of a domain}.
    
Now, let $\Om\subset \RR^N$ be a bounded domain (i.e., an open connected set), whose boundary $\Ga$ is of class $C^{1,1}$. 
Let $u\in W^{1,2}_{0}(\Omega)$ be the (unique) weak solution of
\begin{equation}\label{Anisotropic torsion}
		\begin{cases}
			\Delta_{H}u \coloneqq \div\left(\nabla_{H} u\right) = N, \, & \text{in}\,\, \Omega\\
			u = 0,\, & \text{on}\,\, \Ga,\\
		\end{cases}
	\end{equation}
where 
\begin{equation}\label{anisotropic gradient}
     \nabla_{H}u \coloneqq \nabla_{\xi}V(\nabla u).
    \end{equation}
This means that for any $\varphi\in W^{1,2}_{0}(\Om)$, we have 
\begin{equation}\label{weak sol}
    \int_{\Om}\Big\langle \nabla_{H}u\,,\,\nabla\varphi\Big\rangle\,dx= -N\int_{\Om}\varphi\,dx,
\end{equation}

Given the regularity assumptions on $H$ and on $\Ga$, it follows from~\cite[Theorem $1$]{Lieb} that $u\in C^{1,\beta}\left(\overline{\Om}\right)$ for some $\beta \in (0,1)$, and
\begin{equation}\label{eq:C 1 alpha reg for torsions}
    \left\|u\right\|_{C^{1,\beta}(\overline{\Om})}\leq K_1,
\end{equation}
 where $K_1$ is a positive constant dependent on $N$, the constants $\lambda$ and $\Lambda$ in \eqref{Assump:ellipticity assumption}, the diameter of $\Om$ and the $C^{1,1}$ regularity of $\Ga$.

Since $\Ga$ is assumed to be of class $C^{1,1}$, it satisfies a uniform interior and exterior sphere condition. Since the assumptions set on $H$ imply that $B_{1}^{H_0}$ is uniformly convex, we have that $\Ga$ also satisfies a uniform interior and exterior Wulff sphere condition, i.e.  there exists $r_{\mathcal{W}} >0$ such that for every $x\in \Ga$, and every $0\leq r\leq r_{\mathcal{W}}$ there exists $\overline{x}\in \Omega$ such that 
        \begin{equation}\label{eq:uniform interior Wulff sphere condition}
            B_{r}^{H_0}\left(\overline{x}\right)\subset \Omega \quad \text{and}\quad \overline{B_{r}^{H_0}\left(\overline{x}\right)}\cap \Ga = \{x\}.
        \end{equation}
Moreover, since the radius of the uniform interior and exterior sphere condition may be estimated by constants dependent on the $C^{1,1}$ regularity of $\Ga$, then \eqref{eq:ellipticity of V_0} implies that the radius of the uniform interior and exterior Wulff sphere condition may be estimated by constants dependent on the $C^{1,1}$ regularity of $\Ga$ and the parameters $\lambda$ and $\Lambda$. 

Given two measurable sets $A,B\subset\RR^N$, their distance, $\dist(A,B)$ is defined by 
$$
\dist(A,B) = \inf_{a\in A,\,b\in B}\left|a-b\right|.
$$

The following is a Hopf lemma with an explicit constant which depends only on the radius of the uniform interior Wulff sphere condition. We refer to~\cite{CRS} for an Hopf Lemma for more general anisotropic equations.
 \begin{lem}\label{lem:Hopf lemma for anisotropic torsion}
    Let $H\in C^{2,\alpha}\left(\RR^N\setminus\{0\}\right)$ be a Finsler norm satisfying \eqref{Assump:ellipticity assumption}. Let $\Omega\subset \RR^N$ be a bounded domain with boundary $\Gamma$ of class $C^{1,1}$. Let $u\in W^{1,2}_{0}(\Om)$ be the weak solution of \eqref{Anisotropic torsion}.
    Then $u\leq 0$ in $\Om$ and 
    \begin{equation}\label{eq:explicit Hopf}
    \min_{x\in \Ga}\left| \nabla u(x)\right|\geq \frac{r_{\mathcal{W}}}{\sqrt{\Lambda}},
    \end{equation} 
    where $r_{\mathcal{W}}$ is the radius of the uniform interior Wulff sphere condition of $\Om$, and $\Lambda$ is the parameter appearing in \eqref{Assump:ellipticity assumption}.
    
    In particular,
    $$
    \dist\left(\left\{x\in \overline{\Om}:\,\nabla u(x)=0\right\}\,,\,\Ga\right)>0.
    $$
 \end{lem}
 \begin{proof}
     Substituting $\varphi = \max\{0,u\}$ in \eqref{weak sol} we obtain
     \begin{equation*}
         -N\int_{\{u>0\}} u\,dx= \int_{\{u>0\}}\Big\langle \nabla_{H}u \,,\,\nabla u\Big\rangle\,dx.
     \end{equation*}
     Using \eqref{eq:inner product grad of norm with point} together with the fact that $\nabla_{H}u = H(\nabla u)\nabla_{\xi}H(\nabla u)$, it follows that 
     \begin{equation*}
         0\geq -N\int_{\{u>0\}} u\,dx= \int_{\{u>0\}}H^2(\nabla u)\,dx\geq 0.
     \end{equation*}
     This implies that $u\leq 0$ in $\Om$ since $u\in C^{1,\beta}(\overline{\Om})$. 

     Given $x_0\in \Om$, let $r(x_0) = \min_{y\in\Ga}H_{0}\left(y-x_0\right)$. 
     Consider the function 
     $$
     u_{x_0}(x) = \frac{H_{0}^2(x-x_0)-r(x_0)^2}{2}.
     $$
     Note that $u_{x_0}$ is a solution of the Dirichlet problem
     \begin{equation*}
         \begin{cases}
             \Delta_{H} u_{x_0} = N,&\text{in}\,\,B_{r(x_0)}^{H_0}(x_0)\\
             u_{x_0} = 0,&\text{on}\,\,\partial B_{r(x_0)}^{H_0}(x_0).
         \end{cases}
     \end{equation*}
     Since $u\leq 0$ in $\Om$, it follows that $u\leq u_{x_0}$ in $\partial B_{r(x_0)}^{H_0}(x_0)$, so, by comparison, it follows that 
     \begin{equation}\label{eq:comparison for torsion}
      u(x)\leq \frac{H_{0}^{2}(x-x_{0}) - r(x_0)^2}{2}\quad \text{for all} \quad x\in  B_{r(x_0)}^{H_0}(x_0).
     \end{equation}
     
    Since $\Ga$ satisfies a uniform interior Wulff sphere condition with some radius $r_{\mathcal{W}}$, for every $x\in \Ga$, there exists $\overline{x}\in \Om$ such that 
     \begin{equation*}
    B_{r_W}^{H_0}\left(\overline{x}\right)\subset \Omega \quad \text{and}\quad \overline{B_{r_W}^{H_0}\left(\overline{x}\right)}\cap \Ga = \{x\}. 
\end{equation*}
Using this, \eqref{eq:comparison for torsion} leads to
 \begin{equation}
            \begin{split}
                \partial_{\nu} u(x)&= -\lim_{t\to 0^+}\frac{u(x-t\nu(x))-u(x)}{t}
                \geq \lim_{t\to 0^+}\frac{H_{0}^{2}\left(x-\overline{x}\right)-H_{0}^{2}\left(x-\overline{x}-t\nu(x)\right)}{2 t}\\
                &= H_{0}\left(x-\overline{x}\right)\Big\langle\nabla H_{0}\left(x-\overline{x}\right)\,,\,\nu(x)\Big\rangle =H_{0}\left(x-\overline{x}\right)\left\langle\nabla H_{0}\left(x-\overline{x}\right)\,,\,\frac{\nabla H_{0}(x-\overline{x})}{\left|\nabla H_{0}(x-\overline{x})\right|}\right\rangle\\
                &= H_{0}(x-\overline{x})\left|\nabla H_{0}(x-\overline{x})\right|,
            \end{split}
        \end{equation}
        where, in the last equality, we used the fact that at $x$, $\nu(x)= \frac{\nabla H_{0}(x-\overline{x})}{\left|\nabla H_0 (x-\overline{x})\right|}$.\newline 
Using \eqref{eq:norm equivalency}, \eqref{eq:property 1 grad of norms} and the equality $H_{0}(x-\overline{x}) = r_{\mathcal{W}}$, it follows that 
\begin{equation*}
     \partial_{\nu} u(x) \geq \frac{r_{\mathcal{W}}}{\sqrt{\Lambda}}.
\end{equation*}

In particular, we have 
$$
\min_{x\in\Ga} \left|\nabla u(x)\right|\geq \frac{r_{\mathcal{W}}}{\sqrt{\Lambda}},
$$
and, since $\nabla u\in C^{0,\beta}(\overline{\Om})$ we also have
\begin{equation*}
    \dist\left(\Ga,\mathcal{C}\right)>0,
\end{equation*}
where $\mathcal{C} =\left\{x\in\overline{\Om}:\,\nabla u(x)=0\right\}$.
 \end{proof}

Regarding the second-order regularity of $u$, by~\cite[Proposition $2.3$]{CS} it follows that $u\in W^{2,2}(\Om)$, and by \cite[Theorem $2.5$]{ACCFM} we have that $\nabla_{H} u\in W^{1,2}(\Om)$ and 
\begin{equation*}
    \left\|\nabla_{H} u\right\|_{W^{1,2}(\Om)}\leq K_{2}, 
\end{equation*}
 where $K_{2}$ is a positive constant dependent on $N$, the parameters $\lambda$ and $\Lambda$ that appear in \eqref{Assump:ellipticity assumption}, the measure of $\Om$ and the $C^{1,1}$ regularity of $\Gamma$. We note that, although~\cite[Theorem $2.5$]{ACCFM} is stated in the less general case where $H$ is a norm and $\Ga$ is of class $C^{2,\gamma}$, the result remains valid when $H$ is a Finsler norm (as in our setting) and $\Ga$ is of class $C^{1,1}$. See also the comment after the statement of \cite[Theorem $2.5$]{ACCFM} for more details.
 
The following statement is the anisotropic version of Reilly's identity, proved in \cite[Subsection $3.3$]{DMMN}. 
It is stated and proved in \cite[Subsection $3.3$]{DMMN} for the case $H\in C^{\infty}(\RR^N\setminus\{0\})$ and $\Ga$ is of class $C^{\infty}$. However, an inspection of the proof shows that the result remains valid in the case where $H\in C^{2,\alpha}(\RR^N\setminus\{0\})$ and $\Ga$ is of class $C^{2,\gamma}$ (for some $\gamma\in(0,1)$). The reason for this is that in this setting, due to Lemma~\ref{lem:Hopf lemma for anisotropic torsion}, one applies classical elliptic regularity theory (see, e.g., \cite[Theorem 6.6]{GT} and \cite[Theorem $9.19$]{GT}) to conclude that $u\in C^{2,\beta}(\overline{\Omega_{\delta}})$ for some $\beta\in(0,1)$, where $\Om_{\delta}=\{x\in\Om:\dist(x\,,\,\Ga)<\delta\}$ for some $\delta>0$ small enough.

    \begin{prop}[Proposition $3.3$,\cite{DMMN}]\label{prop: prop 3.3 DMMN}
    Let $H\in C^{2,\alpha}(\RR^N\setminus\{0\})$ be a Finsler norm satisfying \eqref{Assump:ellipticity assumption}.
    Let $\Omega$ be a bounded domain with boundary of class $C^{2,\gamma}$ $($for some $\gamma\in(0,1))$. Let $u\in W^{1,2}_{0}(\Om)$ be the solution of \eqref{Anisotropic torsion}, then 
    \begin{equation*}
        \int_{\Om}\left[N^2-\tr(W^2)\right]\,dx = (N-1)\int_{\Ga} \cM^{H} H(\nabla u)^2 H(\nu)\,d\mathcal{H}^{N-1},
    \end{equation*}
     where $\nu$ denotes the outward unit normal to $\Ga$ and where
     \begin{equation}\label{eq:grad of anisotropic grad}
         W \coloneqq \nabla \left(\nabla_H u\right)= \nabla^{2}_{\xi}V(\nabla u)\nabla^2 u
     \end{equation}
    \end{prop}
    We will rely on this result in Section~\ref{sec:Fundamental identity for ASBT} to establish several integral identities which will serve as the starting points for our quantitative analysis in Sections \ref{Sec: W12 stability for ASBT} and \ref{sec:stability for AHK}.

We conclude this section with the anisotropic analogue of the classical Minkowski identity, which was originally proved for smooth compact hypersurfaces in \cite[Theorem $1.1$]{HL}. However, a standard approximation argument yields the result for compact hypersurfaces of class $C^{1,1}$.
\begin{lem}\label{lemma:Anisotropic Minkowski}
   Let $H\in C^{2,\alpha}(\RR^N\setminus\{0\})$ be a Finsler norm satisfying \eqref{Assump:ellipticity assumption}. Let $\Om\subset \RR^N$ be a bounded domain with boundary $\Ga$ of class $C^{1,1}$. Then, for any $z\in \RR^N$, the following identity holds:
    \begin{equation}\label{eq:Anisotropic Minkowski}
        P_{H}\left(\Om\right)= \int_{\Ga}\cM^{H}\Big\langle x-z\,,\, \nu\Big\rangle\,d\mathcal{H}^{N-1}(x),
    \end{equation}
    where, $P_{H}(\Om)$ denotes the isoperimetric perimeter of $\Om$ as defined in \eqref{eq: def_anisotropic perimeter}, $\cM^H$ denotes the anisotropic mean curvature as defined in \eqref{eq:anisotropic mean curvature} and $\nu$ denotes the outward unit normal to $\Ga$. 
\end{lem}
\section{Fundamental identities for the anisotropic soap bubble theorem and Heintze-Karcher inequality}\label{sec:Fundamental identity for ASBT}

This section is devoted to some fundamental identities that serve as the starting points of our analysis. We note that, in the particular case where $H$ is the Euclidean norm (that is, in the isotropic setting), these identities return those presented in \cite{MP_IUMJ, MP_JAM}.

\begin{prop}\label{prop:Fundamental identity Anisotropic Bubbles}
Let $H\in C^{2,\alpha}\left(\RR^N\setminus\{0\}\right)$ be a Finsler norm satisfying \eqref{Assump:ellipticity assumption}. Let $\Omega\subset \RR^N$ be a bounded with boundary $\Ga$ of class $C^{2,\gamma}$ and let $u\in W^{1,2}_{0}(\Om)$ be a weak solution of \eqref{Anisotropic torsion}.\newline
In this setting, the following identity holds
    \begin{equation}\label{eq:fundamental identity}
    \begin{split}
        &\frac{1}{N-1}\int_{\Om}\left[\tr(W^2) - \frac{\tr(W)^2}{N}\right]\,dx + \mathcal{M}_0^{H}\int_{\Ga}\left|H(\nabla u)-\frac{1}{\mathcal{M}_{0}^{H}}\right|^2H(\nu)\,d\mathcal{H}^{N-1}\\
        &= \int_{\Ga} \left(\mathcal{M}_{0}^{H}-\mathcal{M}^{H}\right) H(\nabla u)^2 H(\nu)\,d\mathcal{H}^{N-1},
    \end{split}
    \end{equation}
    where $W$ is as in \eqref{eq:grad of anisotropic grad}, $\cM_{0}^{H}$ is as in \eqref{eq:reference curvature} and $\nu$ is the outward unit normal to $\Ga$.
\end{prop}
\begin{proof}
Using Proposition \ref{prop: prop 3.3 DMMN} we have that 
\begin{equation}\label{eq:step 1 in fundamental identity}
    \begin{split}
        \frac{1}{N-1}\int_{\Om}\Big[\tr(W^2)-\frac{\tr(W)^2}{N}\Big]\,dx&= N|\Om|-\frac{1}{N-1}\int_{\Om} \Big[N^2-\tr(W^2)\Big]\,dx\\
        &= N|\Om|-\int_{\Ga} \cM^{H} H(\nabla u)^2 H(\nu)\,d\mathcal{H}^{N-1}\\
        &=N|\Om| + \int_{\Ga}\left(\cM_{0}^{H}-\cM^{H}\right)H^2(\nabla u) H(\nu)\,d\mathcal{H}^{N-1}\\
        &- \int_{\Ga}\cM_{0}^{H}H^2(\nabla u)H(\nu)\,d\mathcal{H}^{N-1}.
    \end{split}
\end{equation}

Since for any $x\in\Ga$, $\nu(x) =\frac{\nabla u(x)}{|\nabla u(x)|}$, if follows from \eqref{eq:derivative of norm} and \eqref{Anisotropic torsion} that 
\begin{equation*}
\int_{\Ga} H(\nabla u)H(\nu)\,d\mathcal{H}^{N-1} =  \int_{\Ga}\Big\langle H(\nabla u)\nabla_{\xi}H(\nabla u)\,,\, \nu\Big\rangle d\mathcal{H}^{N-1}= \int_{\Om}\div(\nabla_{H} u)\,dx = N|\Om|.   
\end{equation*}
So, recalling \eqref{eq:reference curvature} we see that
\begin{equation}\label{eq:step 2 in fundamental identity}
    \begin{split}
   &\cM_{0}^{H} \int_{\Ga}\left|H(\nabla u)-\frac{1}{\cM_{0}^{H}}\right|^2 H(\nu)\,d\mathcal{H}^{N-1}\\
   &=\cM_{0}^{H} \int_{\Ga} H^2(\nabla u) H(\nu)\,d\mathcal{H}^{N-1}-2\int_{\Ga}H(\nabla u)H(\nu)\,d\mathcal{H}^{N-1} + \frac{1}{\cM_{0}^{H}}\int_{\Ga} H(\nu)\,d\mathcal{H}^{N-1}\\
   &=\cM_{0}^{H} \int_{\Ga} H^2(\nabla u) H(\nu)\,d\mathcal{H}^{N-1}-N|\Om|. \end{split}
\end{equation}
The result follows by combining \eqref{eq:step 1 in fundamental identity} and \eqref{eq:step 2 in fundamental identity}.
\end{proof}
 
The next theorem provides a fundamental identity for the anisotropic soap bubble theorem. Such an identity will serve as the starting point for our analysis leading to Theorem \ref{thm:W12 stability}. 

\begin{thm}[A fundamental identity for the anisotropic soap bubble theorem]\label{thm:fundamental identity for ASBT}
 Let $H\in C^{2,\alpha}\left(\RR^N\setminus\{0\}\right)$ be a Finsler norm satisfying \eqref{Assump:ellipticity assumption}. Let $\Omega\subset \RR^N$ be a bounded with boundary $\Ga$ of class $C^{2,\gamma}$ $($for some $\gamma\in (0,1))$ and let $u\in W^{1,2}_{0}(\Om)$ be a weak solution of \eqref{Anisotropic torsion}.\\
Let $\cM^{H}_{0}$ be as in \eqref{eq:reference curvature}.

    Then, for any $z\in \RR^N$, the following identity holds
    \begin{equation}\label{eq:unpacked identity}
        \begin{split}
            &\frac{1}{N-1}\int_{\Om}\Big[\tr(W^2)-\frac{\tr(W)^2}{N}\Big]\,dx + \cM_{0}^{H}\int_{\Ga} H(\nu) \left|H(\nabla u)-\frac{1}{\cM_{0}^{H}}\right|^2\,dx\\
            &= \int_{\Ga} \left(\mathcal{M}^{H}-\mathcal{M}_{0}^{H}\right) H(\nabla u) \Big\langle \nabla_{H} u - (x-z)\,,\,\nu\Big\rangle\,d\mathcal{H}^{N-1}\\
        &+ \int_{\Ga}\left(\cM_{0}^{H}-\cM^{H}\right)\left(H(\nabla u)-\frac{1}{\cM_{0}^{H}}\right)\Big\langle x-z\,,\,\nu\Big\rangle\,d\mathcal{H}^{N-1},
        \end{split}
    \end{equation}
where $W$ is as in \eqref{eq:grad of anisotropic grad}, $\nabla_H u$ is as in \eqref{anisotropic gradient} and where $\nu$ denotes the outward unit normal to $\Ga$.
\end{thm}
\begin{proof}
Using the fact $\nu = \frac{\nabla u}{|\nabla u|}$, it follows from \eqref{eq:derivative of norm} and \eqref{eq:inner product grad of norm with point} that
\begin{equation}\label{eq:step 1 in unpacked}
    \begin{split}
        &\int_{\Ga} \left(\mathcal{M}_{0}^{H}-\mathcal{M}^{H}\right) H(\nabla u)^2 H(\nu)\,d\mathcal{H}^{N-1}\\
        &= \int_{\Ga} \left(\mathcal{M}_{0}^{H}-\mathcal{M}^{H}\right) H(\nabla u) \Big\langle H(\nabla u)\nabla_{\xi}H(\nabla u)\,,\,\nu\Big\rangle\,d\mathcal{H}^{N-1}\\
        &=\int_{\Ga} \left(\mathcal{M}_{0}^{H}-\mathcal{M}^{H}\right) H(\nabla u) \Big\langle \nabla_{H}u - (x-z)\,,\,\nu\Big\rangle\,d\mathcal{H}^{N-1}\\
        &+ \int_{\Ga}\left(\cM_{0}^{H}-\cM^{H}\right)\left(H(\nabla u)-\frac{1}{\cM_{0}^{H}}\right)\Big\langle x-z\,,\,\nu\Big\rangle\,d\mathcal{H}^{N-1} + \frac{1}{\cM_{0}^{H}}\int_{\Ga}\left(\cM_{0}^{H}-\cM^{H}\right)\Big\langle x-z\,,\,\nu\Big\rangle\,d\mathcal{H}^{N-1},
    \end{split}
\end{equation}
where $z\in\RR^N$ is arbitrary and $\nabla_{H}u$ is as defined in \eqref{anisotropic gradient}.\\
Using Lemma~\ref{lemma:Anisotropic Minkowski} and recalling the definition of $\cM_{0}^{H}$ we have that 
\begin{equation*}
    \frac{1}{\cM_{0}^{H}}\int_{\Ga}\cM^{H}\Big\langle x-z\,,\,\nu\Big\rangle \,d\mathcal{H}^{N-1}= N|\Om|.
\end{equation*}
On the other hand, the divergence theorem implies that 
\begin{equation*}
    \int_{\Ga} \Big\langle x-z\,,\,\nu \Big\rangle \,d\mathcal{H}^{N-1} = N|\Om|.
\end{equation*}
As a consequence, we follow up on \eqref{eq:step 1 in unpacked} to conclude 
\begin{equation}\label{eq:step 2 in unpacked}
    \begin{split}
        \int_{\Ga} \left(\mathcal{M}_{0}^{H}-\mathcal{M}^{H}\right) H(\nabla u)^2 H(\nu)\,d\mathcal{H}^{N-1}&=\int_{\Ga} \left(\mathcal{M}^{H}-\mathcal{M}_{0}^{H}\right) H(\nabla u) \Big\langle \nabla_{H} u - (x-z)\,,\,\nu\Big\rangle\,d\mathcal{H}^{N-1}\\
        &+ \int_{\Ga}\left(\cM_{0}^{H}-\cM^{H}\right)\left(H(\nabla u)-\frac{1}{\cM_{0}^{H}}\right)\Big\langle x-z\,,\,\nu\Big\rangle\,d\mathcal{H}^{N-1}.
    \end{split}
\end{equation}
The result now follows by combining this and Proposition~\eqref{prop:Fundamental identity Anisotropic Bubbles}.
\end{proof}

The next theorem provides a fundamental identity for the anisotropic Heintze-Karcher inequality. It will serve as the starting point for our analysis leading to Theorem \ref{thm:stability of Heintze-Karcher}.

\begin{thm}[A fundamental identity for the anisotropic Heintze-Karcher inequality]\label{thm:fundamental identity for Heintze-karcher} Let\\ $H\in C^{2,\alpha}\left(\RR^N\setminus\{0\}\right)$ be a Finsler norm satisfying \eqref{Assump:ellipticity assumption}.
    Let $\Om\subset \RR^N$ be a bounded domain with boundary $\Ga$ of class $C^{2,\gamma}$ $($for some $\gamma\in (0,1))$ with $\cM^{H}>0$. Let $u\in W^{1,2}_{0}(\Om)$ be a weak solution to \eqref{Anisotropic torsion}. Then, the following identity holds:
    \begin{equation}
        \frac{1}{N-1}\int_{\Om}\left[\tr(W^2)-\frac{\tr(W)^2}{N}\right]\,dx + \int_{\Ga}\left|1-\cM^{H}H(\nabla u)\right|^2\frac{H(\nu)}{\cM^H}\,d\mathcal{H}^{N-1} = \int_{\Ga}\frac{H(\nu)}{\cM^{H}}\,d\mathcal{H}^{N-1}-N\left|\Om\right|,
    \end{equation}
    where $W$ is as in \eqref{eq:grad of anisotropic grad}.
    In particular, we have 
    \begin{equation}\label{eq:AHK}
        \int_{\Ga}\frac{H(\nu)}{\cM^{H}}\,d\mathcal{H}^{N-1}\geq N\left|\Om\right|.
    \end{equation}
    Moreover, this last inequality remains valid if $\Ga$ is of class $C^{1,1}$.
\end{thm}
\begin{proof}
    Let $u\in W^{1,2}_{0}(\Om)$ be the weak solution to \eqref{Anisotropic torsion}.
    By Proposition \ref{prop:Fundamental identity Anisotropic Bubbles}, we have that 
\begin{equation}\label{eq:step 1 in HK identity}
    \frac{1}{N-1}\int_{\Om}\left[\tr(W^2)-\frac{\tr(W)^2}{N}\right]\,dx = N\left|\Om\right| -\int_{\Ga}\cM^H H(\nabla u)^2H(\nu)\,d\mathcal{H}^{N-1}.
\end{equation}
Noting that 
\begin{equation*}
    \begin{split}
    -\int_{\Ga}\cM^H H(\nabla u)^2 H(\nu)\,d\mathcal{H}^{N-1}&= -\int_{\Ga}\left|1-\cM^{H}H(\nabla u)\right|^2 \frac{H(\nu)}{\cM^{H}}\,d\mathcal{H}^{N-1} + \int_{\Ga}\frac{H(\nu)}{\cM^H}\,d\mathcal{H}^{N-1}\\
    &-2\int_{\Ga} H(\nabla u)H(\nu)\,d\mathcal{H}^{N-1},
    \end{split}
\end{equation*}
we follow up on \eqref{eq:step 1 in HK identity} to obtain 
\begin{equation}\label{eq:step 2 in HK identity}   
\begin{split}
    \frac{1}{N-1}&\int_{\Om}\left[\tr(W^2)-\frac{\tr(W)^2}{N}\right]\,dx + \int_{\Ga}\left|1-\cM^{H}H(\nabla u)\right|^2 \frac{H(\nu)}{\cM^{H}}\,d\mathcal{H}^{N-1}\\
    &=-2\int_{\Ga} H(\nabla u)H(\nu)\,d\mathcal{H}^{N-1}+ N\left|\Om\right| + \int_{\Ga}\frac{H(\nu)}{\cM^H}\,d\mathcal{H}^{N-1}.
    \end{split}
\end{equation}
Since $\nu = \frac{\nabla u}{|\nabla u|}$ on $\Ga$, \eqref{eq:derivative of norm} and the divergence theorem imply that 
\begin{equation*}
    \int_{\Ga} H(\nabla u)H(\nu)\,d\mathcal{H}^{N-1} = \int_{\Ga} \Big\langle H(\nabla u)\nabla_{\xi}H(\nabla u)\,,\,\nu\Big\rangle d\mathcal{H}^{N-1} = \int_{\Om}\div\left(H(\nabla u)\nabla_{\xi}H(\nabla u))\right)\,dx =N|\Om|.
\end{equation*}
The result now follows by substituting this into \eqref{eq:step 2 in HK identity}.

Regarding the last statement, we note that after a standard approximation argument, one may obtain \eqref{eq:AHK} for $\Ga$ of class $C^{1,1}$.
\end{proof}

\section{Proof of Theorem \ref{thm:W12 stability}}\label{Sec: W12 stability for ASBT}

\begin{lem}\label{lem:trace and Poincare type inequality for phi}
Let $H\in C^{2,\alpha}\left(\RR^N\setminus\{0\}\right)$ be a Finsler norm satisfying \eqref{Assump:ellipticity assumption}.
    Let $\Om\subset\RR^N$ be a bounded domain with boundary $\Ga$ of class $C^{1,1}$. Let $u\in W^{1,2}_{0}(\Om)$ be a solution of \eqref{Anisotropic torsion}, $\nabla_H u$ be as in \eqref{anisotropic gradient} and $W$ be as in \eqref{eq:grad of anisotropic grad}. 
    Then, there exists $z\in\RR^N$, such that 
    \begin{equation}\label{eq:estimate for phi}
        \int_{\Ga}\left|\nabla_H u - (x-z)\right|^2\,d\mathcal{H}^{N-1} + \int_{\Om} \left|\nabla_H u - (x-z)\right|^2\,dx\leq C_1 \int_{\Om} \left[\tr(W^2)-\frac{\tr(W)^2}{N}\right]\,dx,
    \end{equation}
    and
    \begin{equation}\label{eq:estimate for inverted phi}
        \int_{\Ga}\left|\nabla u - \nabla V_{0}(x-z)\right|^2\,d\mathcal{H}^{N-1} + \int_{\Om} \left|\nabla u - \nabla V_{0}(x-z)\right|^2\,dx\leq C_1 \int_{\Om} \left[\tr(W^2)-\frac{\tr(W)^2}{N}\right]\,dx,
    \end{equation}
    where $C_1$ is a positive constant dependent only on $N$, the $C^{1,1}$ regularity of $\Ga$, the diameter of $\Om$ and the parameters $\lambda$ and $\Lambda$ appearing in \eqref{Assump:ellipticity assumption}.
\end{lem}
\begin{proof}
    Define 
    \begin{equation}\label{eq:choice of z}
        z \coloneqq \frac{1}{|\Om|}\int_{\Om} \Big(x-\nabla_{H}u\Big)\,dx.
    \end{equation}
    With this definition of $z$, each component of $\nabla_{H}u-(x-z)$ has zero average over $\Om$, so we may apply the Poincaré inequality to each component of $\nabla_{H}u-x$, and use the identity $\nabla\left(\nabla_{H}u-(x-z)\right)= W-I_N$ (which follows from \eqref{eq:grad of anisotropic grad}), to obtain
    \begin{equation}\label{eq:step 1 in ineqs for phi}
        \int_{\Om} \left|\nabla_{H}u -(x-z)\right|^{2}\,dx \leq N C_{P}\int_{\Om} \left|W-I_{N}\right|^2\,dx,
    \end{equation}
    where $C_{P}$ is the best constant in the Poincaré inequality, which can be estimated in terms of $N$, the $C^{0,1}$ regularity of $\Ga$, and the diameter of $\Om$ (see \cite[Remark $3.3$]{MP} for an explicit bound on the constant $C_{P}$).
    
    Following the arguments of the proof of~\cite[Proposition $3.7$]{DMMN}, 
    we see that there exists a constant $C_{N}$ dependent only on $N$, such that 
    \begin{equation}\label{eq:equivalency of hessian deficits}
    \left|W-I_{N}\right|^{2}\leq C_{N}\frac{\Lambda}{\lambda}\left[\tr(W^2)-\frac{\tr(W)^2}{N}\right],
    \end{equation}
    where $W$ is as in \eqref{eq:grad of anisotropic grad}.\\
    Using this, we follow up on \eqref{eq:step 1 in ineqs for phi} to get
    \begin{equation}\label{eq:step 2 in ineqs for phi}
        \int_{\Om} \left|\nabla_{H}u -(x-z)\right|^{2}\,dx \leq NC_{P}C_{N}\frac{\Lambda}{\lambda}\int_{\Om}\left[\tr(W^2)-\frac{\tr(W)^2}{N}\right]\,dx,
    \end{equation}
where $\Lambda$ and $\lambda$ are the parameters that appear in \eqref{Assump:ellipticity assumption}. 

 Employing the notation 
 \begin{equation}\label{eq:phi}
 \phi(x)=\nabla_{H}u(x) -(x-z),
 \end{equation}
 the divergence theorem yields
    \begin{equation}\label{eq:step 3 in ineqs for phi}
    \begin{split}
        \int_{\Ga} \left|\phi(x)\right|^2\,\Big\langle \nabla_{H}u\,,\,\nu\Big\rangle\,d\mathcal{H}^{N-1}(x)&= \int_{\Om} \Delta_{H}u\left|\phi(x)\right|^2\,dx +\int_{\Om} \Big\langle \nabla_{H}u\,,\, \left(W^T-I_{N}\right)\phi(x)\Big\rangle\,dx\\
        &\leq \left(N + \left\|\nabla_{H} u \right\|_{L^{\infty}\left(\Om\right)}\right)\left(\int_{\Om} \left|\phi(x)\right|^2 + \left|\phi(x)\right| \left|W-I_{N}\right|\,dx\right),
    \end{split}
\end{equation}
where, in the second line, we used the fact that $\Delta_{H} u = N$ for a.e. $x\in \Om$ and the Cauchy–Schwarz inequality.

Applying Hölder's inequality, \eqref{eq:step 2 in ineqs for phi} and  \eqref{eq:equivalency of hessian deficits} we follow up on \eqref{eq:step 3 in ineqs for phi} to obtain
\begin{equation}\label{eq:step 4 in ineqs for phi}
    \int_{\Ga} \left|\phi(x)\right|^2\,\Big\langle \nabla_{H}u\,,\,\nu\Big\rangle\,d\mathcal{H}^{N-1}(x) \leq C_{2} \int_{\Om} \left[\tr(W^2)-\frac{\tr(W)^2}{N}\right]\,dx,
\end{equation}
with 
$$
C_{2} = \left(N+\left\|\nabla_{H}u\right\|_{L^{\infty}(\Om)}\right)\left(N^{1/2}C_{P}^{1/2}+NC_{P}\right)C_{N}\frac{\Lambda}{\lambda}.
$$
Using \eqref{eq:norm equivalency}, \eqref{eq:C 1 alpha reg for torsions} we may estimate $\|\nabla_{H}u\|_{L^{\infty}(\Ga)}$ by a constant dependent on $N$, on the parameters $\lambda$ and $\Lambda$ appearing in \eqref{Assump:ellipticity assumption}, the diameter of $\Om$ and the $C^{1,1}$ regularity of $\Ga$.\\
Since, for any $x\in\Ga$, $\nu(x) = \frac{\nabla u(x)}{\left|\nabla u(x)\right|}$, we have that 
$$
\Big\langle \nabla_{H}u(x)\,,\,\nu(x)\Big\rangle = \frac{H(\nabla u(x))^2}{|\nabla u(x)|}\geq \lambda |\nabla u(x)|\geq \frac{\lambda}{\sqrt{\Lambda}} r_{\mathcal{W}}, \quad \text{for any}\,\, x\in\Ga,
$$
where we used \eqref{eq:inner product grad of norm with point}, \eqref{eq:norm equivalency} and Lemma~\ref{lem:Hopf lemma for anisotropic torsion}. 

With this in hand, \eqref{eq:step 4 in ineqs for phi} yields 
\begin{equation}\label{eq:step 5 in ineqs for phi}
        \int_{\Ga} \left|\phi(x)\right|^2\,d\mathcal{H}^{N-1}(x)\leq \frac{\sqrt{\Lambda}}{\lambda}r_{\mathcal{W}}^{-1}  C_{2}\int_{\Om} \left[\tr(W^2)-\frac{\tr(W)^2}{N}\right]\,dx.
\end{equation}
Combining \eqref{eq:step 2 in ineqs for phi} and \eqref{eq:step 5 in ineqs for phi}, we obtain \eqref{eq:estimate for phi} with 
$$
C_{1} = \frac{\sqrt{\Lambda}}{\lambda}r_{\mathcal{W}}^{-1}  C_{2}+ NC_{P}C_{N}\frac{\Lambda}{\lambda}.
$$
As discussed in Section \ref{sec:Preliminaries}, one can estimate $r_{\mathcal{W}}$ using the parameters $\lambda$ and $\Lambda$ appearing in \eqref{Assump:ellipticity assumption} and the $C^{1,1}$ regularity of $\Ga$. This implies that the constant $C_1$ can be estimated in terms of the parameters mentioned in the statement.\\
Finally, up to relabelling $C_{1}$, \eqref{eq:estimate for inverted phi} is a consequence of \eqref{eq:estimate for phi} and the inequality
$$
\left|\nabla u-\nabla V_{0}(x-z)\right|\leq \|\nabla^2 V_{0}\|_{L^{\infty}(\RR^N)}\left|\nabla_{H} u-(x-z)\right|\leq N^{1/2}\lambda^{-1/2}\left|\nabla_{H} u-(x-z)\right|, 
$$
which follows from \eqref{eq:inversion formulas}.
\end{proof}

\begin{rem}[Bounds on $\cM_{0}^{H}$]\label{rem:on the reference mean curvature}
We note that the reference constant $\cM_{0}^{H}$ defined in \eqref{eq:reference curvature} has the following bounds
    $$
    \frac{1}{\sqrt{\Lambda}K_1}\leq \cM_{0}^{H}\leq \frac{\sqrt{\Lambda}}{\sqrt{\lambda}}r_{\mathcal{W}}.
    $$
    In fact, using the divergence theorem, we have
    \begin{equation*}
        N\left|\Om\right| = \int_{\Om} \div\left(\nabla_{H}u\right)\,dx = \int_{\Ga}\Big\langle \nabla_{H}u\,,\,\nu\Big\rangle d\mathcal{H}^{N-1} =  \int_{\Ga}H(\nabla u) H(\nu) d\mathcal{H}^{N-1},
    \end{equation*}
    where we used \eqref{anisotropic gradient}, the fact that $\nu = \frac{\nabla u}{|\nabla u|}$ on $\Ga$, \eqref{eq:derivative of norm} $($with $t= |\nabla u|^{-1})$ and \eqref{eq:inner product grad of norm with point}.\\
    Recalling the definition of $P_{H}(\Om)$ in \eqref{eq: def_anisotropic perimeter}, and using \eqref{eq:norm equivalency}, \eqref{eq:C 1 alpha reg for torsions}, and \eqref{eq:explicit Hopf} we have
    $$
    \frac{\sqrt{\lambda}}{\sqrt{\Lambda}} r_{\mathcal{W}} P_{H}(\Om)\leq N|\Om|\leq \sqrt{\Lambda}K_{1}P_{H}(\Om),
    $$
    where $K_1$ is the constant appearing in \eqref{eq:C 1 alpha reg for torsions} and $r_{\mathcal{W}}$ is the the radius of the uniform interior Wulff sphere condition of $\Om$. The desired bounds on $\cM_{0}^{H}$ easily follow.
\end{rem}

\begin{lem}\label{lem: stability for phi}
Let $H\in C^{2,\alpha}\left(\RR^N\setminus\{0\}\right)$ be a Finsler norm satisfying \eqref{Assump:ellipticity assumption}.
    Let $\Om\subset\RR^N$ be a bounded domain with boundary $\Ga$ of class $C^{1,1}$. Let $u\in W_{0}^{1,2}(\Om)$ be a weak solution of \eqref{Anisotropic torsion}, $\nabla_{H} u$ be as in \eqref{anisotropic gradient} and $\mathcal{M}_{0}^{H}$ be as in \eqref{eq:reference curvature}. Then, there exists $z\in\RR^N$, such that 
    \begin{equation*}
       \left\|H(\nabla u)-\frac{1}{\cM_{0}^{H}}\right\|_{L^{2}(\Ga)} + \left\|\nabla_{H}u-(x-z)\right\|_{L^{2}\left(\Ga\right)}+\left\|\nabla u-\nabla V_{0}(x-z)\right\|_{L^{2}\left(\Ga\right)}\leq C_{3} \left\|\mathcal{M}^{H}-\cM_{0}^{H}\right\|_{L^{2}\left(\Ga\right)},
    \end{equation*}
    where $C_{3}$ is a positive constant dependent only on $N$, the parameters $\Lambda$, $\lambda$ appearing in \eqref{Assump:ellipticity assumption}, the diameter of $\Om$, and the $C^{1,1}$ regularity of $\Ga$.
\end{lem}
\begin{proof}
We start by proving the result under the additional assumption that $\Ga$ is of class $C^{2,\gamma}$ for some $\gamma\in (0,1)$.

Letting $\phi$ be as in \eqref{eq:phi}, and discarding the second term on the left-hand side of \eqref{eq:unpacked identity}, we see, using Hölder's inequality, that
\begin{equation}\label{eq:step 2.1 in optimal W12 stability}
   I\coloneqq \int_{\Om}\left[\tr(W^2)-\frac{\tr(W)^2}{N}\right]\,dx\leq C_{4}\left\|\mathcal{M}^{H}-\cM_{0}^{H}\right\|_{L^2(\Ga)}\left(\left\|\phi\right\|_{L^2(\Ga)} + \left\|H(\nabla u)- \frac{1}{\cM_{0}^{H}}\right\|_{L^2(\Ga)}\right),
\end{equation}
where 
$$
C_{4} \coloneqq (N-1)\left(\|H(\nabla u)\|_{L^{\infty}(\Ga)}+ \|x-z\|_{L^{\infty}(\Ga)}\right).
$$
On the other hand, discarding the first term on the left-hand side of \eqref{eq:unpacked identity} we obtain, via Hölder's inequality, that
\begin{equation}\label{eq:step 2.2 in optimal W12 stability}
    II\coloneqq \int_{\Ga} \left|H(\nabla u)-\frac{1}{\mathcal{M}_{0}^{H}}\right|^2 \,d\mathcal{H}^{N-1}(x)\leq C_{5} \left\|\mathcal{M}^{H}-\cM_{0}^{H}\right\|_{L^2(\Ga)}\left(\left\|\phi\right\|_{L^2(\Ga)} + \left\|H(\nabla u)- \frac{1}{\cM_{0}^{H}}\right\|_{L^2(\Ga)}\right),
\end{equation}
where 
$$
C_{5} \coloneqq \frac{1}{\cM^{H}_{0}}\left(\|H(\nabla u)\|_{L^{\infty}(\Ga)} +\|x-z\|_{L^{\infty}(\Ga)}\right).
$$

Combining \eqref{eq:step 2.1 in optimal W12 stability} and \eqref{eq:step 2.2 in optimal W12 stability} and using Lemma~\ref{lem:trace and Poincare type inequality for phi}, we see that 
\begin{equation}\label{eq:step 3 in optimal W12 stability}
    I + II \leq C_{6}\left\|\cM^{H}-\cM_{0}^{H}\right\|_{L^{2}(\Ga)}\left(I^{1/2}+II^{1/2}\right),
\end{equation}
with 
$$
C_{6} \coloneqq \max\left\{C_{4}\,,\, C_{5} \right\} \,\max\left\{C_{1}^{1/2}\,,\, 1\right\},
$$
where $C_{1}$ is the constant appearing in Lemma~\ref{lem:trace and Poincare type inequality for phi}. Using \eqref{eq:norm equivalency}, \eqref{eq:C 1 alpha reg for torsions}, \eqref{eq:choice of z}, and Remark~\ref{rem:on the reference mean curvature}, the constant $C_{6}$ can be bounded by a constant dependent only on $N$, the parameters $\lambda$ and $\Lambda$ appearing in \eqref{Assump:ellipticity assumption}, the diameter of $\Om$ and the $C^{1,1}$ regularity of $\Ga$.

Therefore, we have 
\begin{equation*}
    \max\left\{I\,,\, II\right\}\leq 2C_{6} \left\|\cM^{H}-\cM_{0}^{H}\right\|_{L^{2}(\Ga)}\max\left\{I\,,\, II\right\}^{1/2},
\end{equation*}
which implies that 
\begin{equation}\label{eq:step 4 in optimal W12 stability}
    \max\left\{I\,,\, II\right\}\leq 4 C_{6}^2 \left\|\cM^{H}-\cM_{0}^{H}\right\|_{L^{2}(\Ga)}^2.
\end{equation}
Combining this with Lemma~\ref{lem:trace and Poincare type inequality for phi}, the result follows with $C_{3} =  2C_{6}\left(2C_{1}^{1/2}+1\right)$. This concludes the proof in the case $\Ga$ is of class $C^{2,\gamma}$ for some $\gamma\in (0,1)$.

Noting that the constant $C_{3}$ above depends only on $N$, the parameters $\Lambda$, $\lambda$ appearing in \eqref{Assump:ellipticity assumption}, the diameter of $\Om$, and the $C^{1,1}$ regularity of $\Om$. 
The result now follows in the general case by a standard approximation argument.
\end{proof}

Now we are ready for the proof of Theorem~\ref{thm:W12 stability} 
\begin{proof}[Proof of Theorem \ref{thm:W12 stability}.]
Let $z\in\RR^N$ be as in Lemma~\ref{lem: stability for phi}. \newline
Due to \eqref{eq: def_anisotropic normal} and \eqref{eq:property 2 grad of norms}, we have that $H_{0}(\nu_{H}) = 1$. So, by the triangle inequality and \eqref{eq:norm equivalency}, we have, for any $x\in \Ga$, that 
\begin{equation}\label{eq:control on radius}
    \left|\frac{1}{\cM_{0}^{H}}\nu_{H}(x)- (x-z)\right|\geq \frac{1}{\sqrt{\lambda}} H_{0}\left(\frac{1}{\cM_{0}^{H}}\nu_{H}(x)- (x-z)\right)\geq \frac{1}{\sqrt{\lambda}} \left|\frac{1}{\cM_{0}^{H}}-H_{0}(x-z)\right|.
\end{equation}
Hence, up to relabelling the constant $C$, proving \eqref{eq:optimal W12 stability} is equivalent to proving 
\begin{equation*}
    \left\|\nu - \frac{\nabla H_0(x-z)}{\left|\nabla H_0(x-z)\right|}\right\|_{L^{2}\left(\Ga\right)}+\left\|\frac{1}{\cM_{0}^{H}}\nu_H - (x-z)\right\|_{L^{2}\left(\Ga\right)}\leq C\left\|\cM^{H}-\cM_{0}^{H}\right\|_{L^{2}\left(\Ga\right)}.
\end{equation*}

Let $u\in W_{0}^{1,2}(\Om)$ be the solution of \eqref{Anisotropic torsion}.
Considering that for any $x\in\Ga$, $\nu(x)= \frac{\nabla u(x)}{|\nabla u(x)|}$, we have, due to \eqref{eq:derivative of norm}, that 
\begin{equation*}
    \begin{split}
    \left|\frac{1}{\cM_{0}^{H}}\nu_H(x) -(x-z)\right|&\leq \left|\frac{1}{\cM_{0}^{H}}\nu_{H}(x) - \nabla_{H}u(x)\right|+\left|\nabla_{H}u(x) - (x-z)\right|\\
    &=\left|\frac{1}{\cM_{0}^{H}}- H(\nabla u)\right|\left|\nu_{H}(x)\right|+\left|\nabla_{H}u(x) - (x-z)\right|\\
    &\leq\sqrt{\Lambda}\left|\frac{1}{\cM_{0}^{H}}- H(\nabla u)\right|+\left|\nabla_{H}u(x) - (x-z)\right|,
    \end{split}
\end{equation*}
where we used the fact that $|\nu_{H}|\leq \sqrt{\Lambda}H_{0}(\nu_H)=\sqrt{\Lambda}$, which is a consequence of~\eqref{eq:norm equivalency} and Lemma~\ref{lem:inversion lemma}.\newline
So, by the triangle inequality and Lemma \ref{lem: stability for phi}, we have that 
\begin{equation}\label{eq:stability for anisotropic normal} 
    \begin{split}
    \left\|\frac{1}{\cM_{0}^{H}}\nu_H(x) -(x-z)\right\|_{L^{2}(\Ga)}&\leq \sqrt{\Lambda}\left\|\frac{1}{\cM_{0}^{H}}- H(\nabla u)\right\|_{L^{2(\Ga)}}+\left\|\nabla_{H}u(x) - (x-z)\right\|_{L^{2}(\Ga)}\\
    &\leq \left(\sqrt{\Lambda}+1\right)C_{3}\left\|\cM_{0}^{H}-\cM^{H}\right\|_{L^{2}(\Ga)}. 
    \end{split}
\end{equation}

From Lemma~\ref{lem: stability for phi}, we know that   
\begin{equation}\label{eq:application of lem:stability for phi}
    \left\|\nabla u-\nabla V_{0}(x-z)\right\|_{L^{2}\left(\Ga\right)}\leq C_{3} \left\|\mathcal{M}^{H}-\cM_{0}^{H}\right\|_{L^{2}\left(\Ga\right)}.
\end{equation}
Using the triangle inequality and \eqref{eq:explicit Hopf}, we have, for any $x\in\Ga$ such that  $x\neq z$,
that
\begin{equation}\label{eq:pointwise control on normal}
    \begin{split}
        \left|\nu-\frac{\nabla V_{0}(x-z)}{|\nabla V_{0} (x-z)|}\right|&\leq \left|\frac{\nabla u}{|\nabla u|}-\frac{\nabla V_{0}(x-z)}{|\nabla u|}\right| +\left|\frac{\nabla V_{0}(x-z)}{|\nabla u|}-\frac{\nabla V_{0}(x-z)}{|\nabla V_{0} (x-z)|}\right|\\
        &\leq \frac{\sqrt{\Lambda}}{r_{\mathcal{W}}}\left|\nabla u-\nabla V_{0}(x-z)\right|+\frac{\left|\nabla V_{0}(x-z)\right|}{\left|\nabla u\right|\left|\nabla V_{0}(x-z)\right|}\left|\left|\nabla u\right|-\left|\nabla V_{0}(x-z)\right|\right|\\
        &\leq 2\frac{\sqrt{\Lambda}}{r_{\mathcal{W}}}\left|\nabla u-\nabla V_{0}(x-z)\right|.
    \end{split}
\end{equation}
As discussed in Remark \ref{rem:location of z}, one could note that if the deviation $\left\|\cM^{H}-\cM_{0}^{H}\right\|_{L^{2}(\Ga)}$ is small enough then we must have $z\in\Omega$ and hence $x\neq z$ for any $x \in \Gamma$. Here, we do not need this, and we are not imposing any smallness assumption.

Applying \eqref{eq:application of lem:stability for phi} we then see that 
\begin{equation}\label{eq:step 1 in control of normal}
        \left\|\nu-\frac{\nabla V_{0}(x-z)}{|\nabla V_{0} (x-z)|}\right\|_{L^{2}(\Ga)}\leq 2\frac{\sqrt{\Lambda}}{r_{\mathcal{W}}}\left\|\nabla u-\nabla V_{0}(x-z)\right\|_{L^{2}(\Ga)}\leq 2\frac{\sqrt{\Lambda}}{r_{\mathcal{W}}}C_{3}\left\|\cM^{H}-\cM_{0}^{H}\right\|_{L^{2}(\Ga)}.
\end{equation}

The result now follows by combining \eqref{eq:stability for anisotropic normal} and \eqref{eq:step 1 in control of normal}.
\end{proof}
\begin{rem}[On the location of $z$ in Theorem~\ref{thm:W12 stability}]\label{rem:location of z}
Even though our arguments do not need it, one could show that if $\left\|\cM^{H}-\cM_{0}^{H}\right\|_{L^{2}(\Ga)}$ is small enough, then $z\in \Om$. To see this, note that using \eqref{eq:stability for anisotropic normal} and leveraging the arguments in Lemma~\ref{lem:qualitative lemma AHK}, one concludes that $\left\|\frac{1}{\cM_{0}^{H}}\nu_H-(x-z)\right\|_{L^{\infty}(\Ga)}<\frac{\sqrt{\lambda}}{2\cM_{0}^{H}}$, provided that $\left\|\cM^{H}-\cM_{0}^{H}\right\|_{L^{2}(\Ga)}$ is small enough. Using \eqref{eq:inner product grad of norm with point} and \eqref{eq:norm equivalency} we find, for any $x\in\Ga$, that $\Big\langle \nu_H (x)\,,\,\nu(x)\Big\rangle = H(\nu)> \sqrt{\lambda}$. On the other hand, since $\left\|\frac{1}{\cM_{0}^{H}}\nu_H-(x-z)\right\|_{L^{\infty}(\Ga)}<\frac{\sqrt{\lambda}}{2\cM_{0}^{H}}$, then it follows, for any $x\in\Ga$, that
$$
\Big\langle \nu(x) \,,\, x-z\Big\rangle = \Big\langle \nu(x) \,,\, x-z - \frac{1}{\cM_{0}^{H}}\nu_H(x)\Big\rangle + \Big\langle \nu(x) \,,\, \nu_H(x)\Big\rangle > -\frac{\sqrt{\lambda}}{2\cM_{0}^{H}}+\frac{\sqrt{\lambda}}{\cM_{0}^{H}}>0.
$$
This implies that $\Om$ is starshaped around $z$, and therefore $z\in\Om$.
\end{rem}


\section{Proof of Theorem~\ref{thm:stability of Heintze-Karcher}}\label{sec:stability for AHK}

\begin{lem}\label{lem: stability for gradient in AHK}
Let $H\in C^{2,\alpha}\left(\RR^N\setminus\{0\}\right)$ be a Finsler norm satisfying \eqref{Assump:ellipticity assumption}.
    Let $\Om\subset\RR^N$ be a bounded domain with boundary $\Ga$ of class $C^{1,1}$ such that $\cM^{H}>0$. Let $u\in W_{0}^{1,2}(\Om)$ be a weak solution of \eqref{Anisotropic torsion}. Then
    \begin{equation*}
    \left\|\nabla u-\nabla V_{0}(x-z)\right\|_{L^{2}\left(\Omega\right)}+\left\|\nabla u-\nabla V_{0}(x-z)\right\|_{L^{2}\left(\Ga\right)}\leq C_{7} \left(\int_{\Ga}\frac{H(\nu)}{\cM^{H}}d\mathcal{H}^{N-1}- N\left|\Om\right|\right)^{1/2},
    \end{equation*}
    where $z\in \RR^N$ is as in Lemma~\ref{lem:trace and Poincare type inequality for phi} and where $C_{7}$ is a positive constant dependent only on $N$, the parameters $\Lambda$, $\lambda$ appearing in \eqref{Assump:ellipticity assumption}, the diameter of $\Om$, and the $C^{1,1}$ regularity of $\Ga$.
\end{lem}
\begin{proof}
    For now, assume that $\Ga$ is of class $C^{2,\gamma}$ for some $\gamma\in (0,1)$.
Due to this extra assumption, we may combine Lemma~\ref{lem:trace and Poincare type inequality for phi} and Theorem~\ref{thm:fundamental identity for Heintze-karcher} to find that 
    \begin{equation}
        \int_{\Ga}\left|\nabla u - \nabla V_{0}(x-z)\right|^2\,d\mathcal{H}^{N-1} + \int_{\Om} \left|\nabla u - \nabla V_{0}(x-z)\right|^2\,dx\leq (N-1)C_1\left(\int_{\Ga}\frac{H(\nu)}{\cM^{H}}\,d\mathcal{H}^{N-1}-N\left|\Om\right|\right),
    \end{equation}
    where $C_{1}$ is the constant appearing in Lemma~\ref{lem:trace and Poincare type inequality for phi}, which depends only on $N$, the parameters $\lambda$ and $\Lambda$ that appear in \eqref{Assump:ellipticity assumption}, the diameter of $\Om$ and the $C^{1,1}$ regularity of $\Ga$. 

    Since the constant $C_{1}$ does not depend on the $C^{2,\gamma}$ regularity of $\Ga$, the result for $\Ga$ of class $C^{1,1}$ will follow with a standard approximation argument.
\end{proof}

 \begin{lem}\label{lem:qualitative lemma AHK}
    Let $H\in C^{2,\alpha}\left(\RR^N\setminus\{0\}\right)$ be a Finsler norm satisfying \eqref{Assump:ellipticity assumption}.
    Let $\Om\subset\RR^N$ be a bounded domain with boundary $\Ga$ of class $C^{1,1}$ such that $\cM^{H}>0$. Let $u\in W_{0}^{1,2}(\Om)$ be a weak solution of \eqref{Anisotropic torsion}. Let $z\in \RR^N$ be as in Lemma~\ref{lem:trace and Poincare type inequality for phi}.
    For every $\delta > 0$, there exists $\overline{\delta}>0$ dependent only on $\delta$, $N$, the parameters $\lambda$ and $\Lambda$ appearing in \eqref{Assump:ellipticity assumption}, the diameter of $\Om$ and the $C^{1,1}$ regularity of $\Ga$, such that, if
    \begin{equation*}
    \int_{\Ga}\frac{H(\nu)}{\cM^{H}}\,d\mathcal{H}^{N-1}-N\left|\Om\right|\leq \overline{\delta},
\end{equation*}
then
\begin{equation*}
    \max_{x\in \Ga} \left|\nabla u - \nabla V_{0}(x-z)\right|\leq \delta.
\end{equation*}
\end{lem}
\begin{proof}
 From Lemma~\ref{lem: stability for gradient in AHK} we have
    \begin{equation}\label{eq:step 1 in qualitative AHK}
        \left\|\nabla u-\nabla V_{0}(x-z)\right\|_{L^{2}\left(\Ga\right)}\leq C_{7} \left(\int_{\Ga}\frac{H(\nu)}{\cM^{H}}\,d\mathcal{H}^{N-1}-N\left|\Om\right|\right)^{1/2}.
    \end{equation}
    Given $\delta>0$, define the measurable set
$$
S_{\delta} \coloneqq \left\{x\in \Ga:\,\left|\nabla u - \nabla V_{0}(x-z)\right|^2>\delta \right\}.
$$
Assume that there exists $y_{0} \in S_{\delta}$. Then, given that $\nabla u - \nabla V_{0}(x-z)\in C^{0,\beta}(\overline{\Om})$, there exists $r>0$ dependent only on the $C^{0,\beta}$ regularity of $\nabla u - \nabla V_{0}(x-z)$, the $C^{1,1}$ regularity of $\Ga$ and on $\delta$, such that 
\begin{equation*}
    \left\{y\in\Ga:\, |y-y_0|<r\right\}\subset S_{\delta/2}, \quad \text{and} \quad \mathcal{H}^{N-1}\left(\left\{y\in\Ga:\, |y-y_0|<r\right\}\right)\geq \left|B_{1}^{N-1}\right| r^{N-1},
\end{equation*}
where $\left|B_{1}^{N-1}\right|$ denotes the $(N-1)-$dimensional Lebesgue measure of the unit ball in $\RR^{N-1}$.\newline 
The definition of $S_{\delta/2}$ and \eqref{eq:step 1 in qualitative AHK} now yield
\begin{equation*}
   \left|B_{1}^{N-1}\right| r^{N}\leq \mathcal{H}^{N-1}\left(S_{\delta/2}\right)\leq \frac{2}{\delta}\int_{\Ga}\left|\nabla u-\nabla V_{0}(x-z)\right|^2\,\mathcal{H}^{N-1}\leq \frac{2}{\delta}C_{7}^{2} \left(\int_{\Ga}\frac{H(\nu)}{\cM^{H}}\,d\mathcal{H}^{N-1}-N\left|\Om\right|\right).
\end{equation*}
So, setting $\overline{\delta}\coloneqq \frac{\delta}{2} \frac{\left|B_{1}^{N-1}\right|}{C_{7}^2}r^{N}$, we find that if
\begin{equation*}
    \int_{\Ga}\frac{H(\nu)}{\cM^{H}}\,d\mathcal{H}^{N-1}-N\left|\Om\right| \leq \overline{\delta},
\end{equation*}
then $S_{\delta}=\emptyset$.
\end{proof}

Now we are ready for the proof of Theorem~\ref{thm:stability of Heintze-Karcher}.
\begin{proof}[Proof of Theorem~\ref{thm:stability of Heintze-Karcher}]
Let $u\in W^{1,2}_{0}(\Om)$ be the weak solution of \eqref{Anisotropic torsion}. Let $z\in \RR^N$ be as in Lemma~\ref{lem:trace and Poincare type inequality for phi}. We break the proof into two steps. 

\textbf{Step 1:}
In this step, we estimate the first term on the left-hand side of \eqref{eq:W 1,2 stability for HK}, namely, we prove the inequality
\begin{equation}\label{eq:step 1 in W12 AHK}
    \int_{\Ga} \left|H_{0}(x-z)-R\right|^2\,d\mathcal{H}^{N-1}\leq C_{8} \left(\int_{\Ga}\frac{H(\nu)}{\cM^{H}}d\mathcal{H}^{N-1}-N\left|\Om\right|\right),
\end{equation}
for some $R>0$, and where $C_{8}$ is a positive constant dependent only on $N$, the parameters $\lambda$ and $\Lambda$ appearing in \eqref{Assump:ellipticity assumption}, the diameter of $\Om$ and the $C^{1,1}$ regularity of $\Ga$. 

Define $R = \sqrt{2\rho}$, where $\rho$ is given by
\begin{equation}\label{eq:choice of radius for AHK}
    \rho = \frac{1}{|\Om|}\int_{\Om}\left[\frac{1}{2}H_{0}^{2}(x-z) - u\right]\,dx.
\end{equation}
Since $\Om$ is of class $C^{1,1}$ it satisfies a uniform interior sphere condition with radius $r_{\Om}$. As a consequence, we see that 
$$
\rho \geq \frac{1}{2}\int_{\Om} H_{0}^2(x-z)\,dx \geq \frac{1}{2\Lambda} \int_{B_{r_{\Om}}(x_0)}\left|x-z\right|^2\,dx\geq \frac{\mathcal{H}^{N-1}\left(\SS^{N-1}\right)}{2\Lambda (N+2)} r_{\Om}^{N+2}.
$$
As a consequence of this, we see that 
\begin{equation}\label{eq:step 1 in L2 stab for AHK}
    \left|H_{0}(x-z) - R\right| = \frac{2}{H_0(x-z) + R}\left|\frac{1}{2}H_{0}^{2}(x-z)-\rho\right|\leq C_{9} \left|V_{0}(x-z)-\rho\right|, 
\end{equation}
where $C_{9}$ is a constant dependent only on $N$, the parameter $\Lambda$ that appears in \eqref{Assump:ellipticity assumption} and the $C^{1,1}$ regularity of $\Om$.

Using \eqref{eq:derivative of norm} and \eqref{eq:norm equivalency} and Lemma~\ref{lem:Hopf lemma for anisotropic torsion} we have that 
$$
\Big\langle \nabla_{H}u(x)\,,\,\nu(x)\Big\rangle = \frac{H(\nabla u(x))^2}{|\nabla u(x)|}\geq \lambda |\nabla u(x)|\geq \frac{\lambda}{\sqrt{\Lambda}} r_{\mathcal{W}}.
$$ 
It follows from this and \eqref{eq:step 1 in L2 stab for AHK} that 
\begin{equation*}
        \int_{\Ga} \left| H_{0}(x-z)-R\right|^{2}  \,d\mathcal{H}^{N-1}\leq C_{10}\int_{\Ga}\left|V_{0}(x-z)-u-\rho\right|^{2}\,\Big\langle \nabla_{H}u\,,\,\nu\Big\rangle\,d\mathcal{H}^{N-1},
\end{equation*}
where 
$$
C_{10} = C_{9}^2 \frac{\sqrt{\Lambda}}{\lambda}r_{\mathcal{W}}^{-1}.
$$
The divergence theorem implies that 
\begin{equation*}
    \begin{split}
        \int_{\Ga}\left| H_{0}(x-z)-R\right|^{2}  \,d\mathcal{H}^{N-1}&\leq C_{10}\int_{\Om} \Delta_{H} u \left|V_{0}^{2}(x-z)-u-\rho\right|^2\,dx\\
        &+ C_{10}\int_{\Om} \left(V_{0}(x-z)-u-\rho\right)\Big\langle \nabla_{H}u\,,\, \nabla V_{0}(x-z) -\nabla u\Big\rangle \,dx.
    \end{split}
\end{equation*}
Applying the Cauchy-Schwarz inequality and Hölder's inequality, we follow up on the previous equation to obtain 
\begin{equation}\label{eq:step 2 in L2 stab for AHK}
    \begin{split}
        \int_{\Ga}\left|H_{0}(x-z)-R\right|^{2}  d\mathcal{H}^{N-1}&\leq C_{10}N\int_{\Om}\left|V_{0}(x-z)-u-\rho\right|^2\,dx\\
        &+ C_{10} \left\|\nabla_{H} u \right\|_{L^{\infty}\left(\Om\right)}\int_{\Om}\left|V_{0}(x-z)-u-\rho\right|\left|\nabla V_{0}(x-z)-\nabla u \right|\,dx\\
        &\leq C_{10}N\int_{\Om}\left|V_{0}(x-z)-u-\rho\right|^2\,dx\\
        &+C_{10} \left\|\nabla_{H} u \right\|_{L^{\infty}\left(\Om\right)}\|\nabla V_0(x-z) -\nabla u\|_{L^{2}(\Om)}\left\|V_{0}(x-z)-u\right\|_{L^{2}(\Om)}.
    \end{split}
\end{equation}
Since, by \eqref{eq:choice of radius for AHK}, $V_{0}(x-z)-u-\rho$ has zero average in $\Om$, we apply the Poincaré inequality to obtain
\begin{equation*}
    \int_{\Om}\left|V_{0}(x-z)-u-\rho\right|^2\,dx\leq C_{P} \int_{\Om}\left|\nabla V_{0}(x-z)-\nabla u\right|^2\,dx,
\end{equation*}
where the constant $C_P$ is the best constant in the Poincare inequality and can be estimated in terms of $N$, the $C^{0,1}$ regularity of $\Ga$ and the diameter of $\Om$ (see \cite[Remark $3.3$]{MP} for an explicit bound on the constant $C_{P}$).\\
With this in hand, we follow up on \eqref{eq:step 2 in L2 stab for AHK} to obtain
\begin{equation*}
        \int_{\Ga}\left|H_{0}(x-z)-R\right|^{2}  d\mathcal{H}^{N-1}   \leq C_{10}\left(NC_{P} + \left\|\nabla_{H} u \right\|_{L^{\infty}\left(\Om\right)} C_{P}^{1/2}\right)\int_{\Om}\left|\nabla V_{0}(x-z)-\nabla u\right|^2\,dx.
\end{equation*}
Using Lemma~\ref{lem: stability for gradient in AHK} we obtain \eqref{eq:step 1 in W12 AHK} with
$$C_{8} =C_{10}\left(NC_{P} + \left\|\nabla_{H} u \right\|_{L^{\infty}\left(\Om\right)} C_{P}^{1/2}\right).$$
Note that the constant $C_{8}$ depends only on $N$, the parameters $\lambda$, $\Lambda$ appearing in \eqref{Assump:ellipticity assumption}, the diameter of $\Om$ and the $C^{1,1}$ regularity of $\Ga$. 
This is easily checked by tracking the dependencies of the constant $C_{10}$, using \eqref{eq:C 1 alpha reg for torsions} and recalling the dependencies of the constant $C_P$.

\textbf{Step 2:} In this step, we estimate the second term on the left-hand side of \eqref{eq:W 1,2 stability for HK}, that is we prove the inequality
\begin{equation}\label{eq:step 2 in W12 AHK}
    \int_{\Ga} \left|R\nu_{H}-(x-z)\right|^2\,d\mathcal{H}^{N-1}\leq C_{11} \left(\int_{\Ga}\frac{H(\nu)}{\cM^{H}}d\mathcal{H}^{N-1}-N\left|\Om\right|\right),
\end{equation}
for some positive constant $C_{11}$ dependent only on $N$, the parameters $\lambda$ and $\Lambda$ appearing in \eqref{Assump:ellipticity assumption}, the diameter of $\Om$ and the $C^{1,1}$ regularity of $\Ga$.

First, we note that by the triangle inequality we have the following
\begin{equation}\label{eq:ineq for anisotropic normal}
        \left|R\nu_{H}-(x-z)\right|\leq \left|\nu_{H}\right|\left|R-H_{0}(x-z)\right| + H_{0}(x-z)\left|\nu_{H}-\frac{(x-z)}{H_{0}(x-z)}\right|.
\end{equation}
Using Lemma~\ref{lem:inversion lemma} we find that 
\begin{equation}\label{eq:x-z in terms of V and V_0}
    \frac{x-z}{H_{0}(x-z)} = \frac{\nabla_{\xi}V(\nabla V_{0}(x-z))}{H_{0}(x-z)}= \nabla_{\xi}H(\nabla V_{0}(x-z)).
\end{equation}

Applying Lemma~\ref{lem:qualitative lemma AHK} with $\delta \coloneqq\frac{r_{\mathcal{W}}}{4\sqrt{\Lambda}}$, we find that there exists $\overline{\delta}>0$ dependent only, $N$, the parameters $\lambda$ and $\Lambda$ appearing in \eqref{Assump:ellipticity assumption}, the diameter of $\Om$ and the $C^{1,1}$ regularity of $\Ga$, such that if 
\begin{equation}\label{eq:suff condition for anisotropic normal control AHK}
    \int_{\Ga}\frac{H(\nu)}{\cM^{H}}d\mathcal{H}^{N-1}-N\left|\Om\right|\leq \overline{\delta}, 
\end{equation}
then we have that 
\begin{equation*}
    \left\|\nabla u - \nabla V_{0}(x-z)\right\|_{L^{\infty}(\Ga)}\leq \frac{r_{\mathcal{W}}}{4\sqrt{\Lambda}} .
\end{equation*}
As a consequence, Lemma~\ref{lem:Hopf lemma for anisotropic torsion} implies that 
\begin{equation*}
    \min_{x\in \Ga} \left\{\left|\nabla u(x)\right|\,,\, \left|\nabla V_{0}(x-z)\right|\right\}\geq \frac{3 r_{\mathcal{W}}}{4 \sqrt{\Lambda}}.
\end{equation*}
This implies that for any $x\in \Ga$, the straight line connecting $\nabla u(x)$ and $\nabla V_{0}(x-z)$ is contained in $\RR^N\setminus B_{r}(0)$, where $r = \frac{3r_{\mathcal{W}}}{4\sqrt{\Lambda}}$.\\
Using the relation $\nu_H = \nabla_{\xi}H(\nabla u)$ (which easily follows by \eqref{eq:derivative of norm} and the definition of $\nu_H$ in \eqref{eq: def_anisotropic normal}), \eqref{eq:x-z in terms of V and V_0}, and the fundamental theorem of calculus, we find
\begin{equation*}
    \left|\nu_{H}-\frac{x-z}{H_{0}(x-z)}\right|= \left|\nabla_{\xi}H(\nabla u)-\nabla_{\xi}H(\nabla V_{0}(x-z))\right|\leq \left\|\nabla_{\xi}^2 H\right\|_{L^{\infty}(\RR^N\setminus B_{r}(0))} \left|\nabla u - \nabla V_{0}(x-z)\right|.
\end{equation*}
Substituting this into \eqref{eq:ineq for anisotropic normal}, we find that 
\begin{equation*}
    \left|R\nu_{H}-(x-z)\right|\leq C_{12} \Big(\left|R-H_{0}(x-z)\right| + \left|\nabla u-\nabla V_{0}(x-z)\right|\Big),
\end{equation*}
where 
$$
C_{12} = \max\left\{\left|\nu_{H}\right|\,,\, \left\|H_{0}(x-z)\right\|_{L^{\infty}(\Om)}\left\|\nabla_{\xi}^2 H\right\|_{L^{\infty}(\RR^N\setminus B_{r}(0))}\right\}.
$$
So, using \eqref{eq:step 1 in W12 AHK} and Lemma~\ref{lem: stability for gradient in AHK} we find that 
\begin{equation*}
    \left\|R\nu_{H}-(x-z)\right\|_{L^{2}(\Ga)}\leq C_{12}\left(C_{8}^{1/2}+C_{7}^{1/2}\right)\left(\int_{\Ga}\frac{H(\nu)}{\cM^{H}}d\mathcal{H}^{N-1}-N\left|\Om\right|\right)^{1/2}.
\end{equation*}
This proves \eqref{eq:step 2 in W12 AHK} under the additional assumption \eqref{eq:suff condition for anisotropic normal control AHK}.

On the other hand, if \eqref{eq:suff condition for anisotropic normal control AHK} is not in force, then we have that 
\begin{equation*}
\int_{\Ga} \left|R\nu_{H}-(x-z)\right|d\mathcal{H}^{N-1}\leq \frac{\left\|R\nu_{H}-(x-z)\right\|_{L^{\infty}(\Ga)}\mathcal{H}^{N-1}\left(\Ga\right)}{\overline{\delta}}\left(\int_{\Ga}\frac{H(\nu)}{\cM^{H}}d\mathcal{H}^{N-1}-N\left|\Om\right|\right).
\end{equation*}

In any case, \eqref{eq:step 2 in W12 AHK} we see that holds with 
$$
C_{11} = \max\left\{C_{12}^{2}\left(C_{8}^{1/2}+C_{7}^{1/2}\right)^{2}\,,\,\frac{\left\|R\nu_{H}-(x-z)\right\|_{L^{\infty}(\Ga)}\mathcal{H}^{N-1}\left(\Ga\right)}{\overline{\delta}}\right\}.
$$

\textbf{Step 3:}In this step, we finalise the proof by estimating the third term on the left-hand side of \eqref{eq:W 1,2 stability for HK}.
Proceeding like in \eqref{eq:pointwise control on normal} we find that 
\begin{equation*}
    \left\|\nu-\frac{\nabla V_{0}(x-z)}{|\nabla V_{0} (x-z)|}\right\|_{L^{2}(\Ga)}\leq 2\frac{\sqrt{\Lambda}}{r_{\mathcal{W}}}\left\|\nabla u-\nabla V_{0}(x-z)\right\|_{L^{2}(\Ga)}.
\end{equation*}
An application of Lemma~\ref{lem: stability for gradient in AHK} thus leads to 
\begin{equation}\label{eq:step 1 in control of normal in AHK}
    \left\|\nu-\frac{\nabla V_{0}(x-z)}{|\nabla V_{0} (x-z)|}\right\|_{L^{2}(\Ga)}\leq 2\frac{\sqrt{\Lambda}}{r_{\mathcal{W}}}C_{7}\left(\int_{\Ga}\frac{H(\nu)}{\cM^{H}}d\mathcal{H}^{N-1}-N\left|\Om\right|\right)^{1/2}.
\end{equation}
Combining \eqref{eq:step 1 in W12 AHK}, \eqref{eq:step 2 in W12 AHK} and \eqref{eq:step 1 in control of normal in AHK} yields \eqref{eq:W 1,2 stability for HK} with 
\begin{equation*}
C = \sqrt{C_{8}} + \sqrt{C_{11}} + 2\frac{\sqrt{\Lambda}}{r_{\mathcal{W}}}C_{7}.\qedhere
\end{equation*}
\end{proof}

\begin{rem}[On the location of $z$ in Theorem \ref{thm:stability of Heintze-Karcher}]
The arguments in Remark~\ref{rem:location of z} can be adapted in the setting of Theorem~\ref{thm:stability of Heintze-Karcher} to show that if $\displaystyle\int_{\Ga}\frac{H(\nu)}{\cM^H}d\mathcal{H}^{N-1}-N\left|\Om\right|$ is small enough, then $z\in\Om$.
The only difference is that here one uses \eqref{eq:step 2 in W12 AHK} instead of \eqref{eq:stability for anisotropic normal}.
\end{rem}

\section*{Acknowledgements}
	This work has been supported by the Australian Research Council (ARC) Discovery Early Career Researcher Award (DECRA) DE230100954 ``Partial Differential Equations: geometric aspects and applications''.
	Giorgio Poggesi is a member of the Australian Mathematical Society (AustMS).
    João Gonçalves da Silva is supported by a School of Mathematical Sciences Scholarship at Adelaide University.


\end{document}